\documentclass[11pt,reqno]{amsart}
\usepackage{amssymb}
\usepackage{amsthm}
\usepackage{eucal}
\newcommand{\R}{\mathbb R}
\newcommand{\C}{\mathbb C}
\newcommand{\N}{\mathbb N}
\newcommand{\Z}{\mathbb Z}
\newcommand{\T}{\mathbb T}

\newcommand{\tres}{|\!|\!|}

\newcommand{\lan}{\langle}
\newcommand{\ran}{\rangle}
\newtheorem{theorem}{Theorem}[section]
\newtheorem{proposition}[theorem]{Proposition}
\newtheorem{remark}[theorem]{Remark}
\newtheorem{lemma}[theorem]{Lemma}
\newtheorem{corollary}[theorem]{Corollary}
\newtheorem{definition}[theorem]{Definition}

\begin{document}

\title[Control and Stabilization of the Benjamin-Ono Equation]{Control and Stabilization of the Benjamin-Ono Equation on a Periodic Domain}
\author{Felipe Linares}
\address{Instituto de Matematica Pura e Aplicada,
Estrada Dona Castorina 110, Rio de Janeiro 22460-320, Brazil}
\email{linares@impa.br}
%\thanks{}

\author{Lionel Rosier}
\address{Institut Elie Cartan, UMR 7502 UHP/CNRS/INRIA,
B.P. 70239, 54506 Van\-d\oe uvre-l\`es-Nancy Cedex, France}
\email{rosier@iecn.u-nancy.fr}



\begin{abstract} It was proved by Linares and Ortega in \cite{LO} that the {\em linearized} Benjamin-Ono equation
posed on a periodic domain $\T$ with a distributed control supported on an arbitrary subdomain is exactly controllable and exponentially stabilizable.
The aim of this paper is to extend those results to the {\em full} Benjamin-Ono equation.
%Using the theory developed by Molinet in \cite{molinet}, we first show that the equation with
%a forcing term is globally well-posed in the space $L^2_0(\T)$ of square integrable functions with zero mean over $\T$,
%with a flow map of class $C^1$. This enables us to derive the local controllability
%$in $L^2_0(\T )$.
A feedback law in the form of a localized damping is incorporated in the equation. A smoothing effect established with the aid of a  propagation of regularity
property is used to prove the semi-global stabilization in $L^2(\T )$ of weak solutions obtained by the method of vanishing viscosity. The local
well-posedness and the local exponential stability in $H^s(\T )$
are also established for  $s>1/2$ by using the contraction mapping theorem. Finally, the local exact controllability is derived in $H^s(\T )$
for $s>1/2$ by combining the above feedback law with some open loop control.
\end{abstract}

\maketitle
\section{Introduction}
The Benjamin-Ono (BO) equation can we written as
\[
u_t + {\mathcal H} u_{xx} + uu_x=0,
\]
where $u=u(x,t)$ denotes a real-valued function of the variables $x\in\R$ and
$t\in \R$, and ${\mathcal H}$ denotes the Hilbert transform defined as
\[
\widehat {{\mathcal H } u} (\xi )= -i \, \text{sgn} (\xi )\, \hat u(\xi ).
\]
This integro-differential equation models the propagation of internal waves in stratified fluids of great depth (see
\cite{benjamin,ono}) and turns out to be important in other physical situations as well (see \cite{DR,ishimori,MK}).
Among noticeable properties of this equation we can mention that: (i) it defines a Hamiltonian system; (ii) it admits infinitely
many conserved quantities (see \cite{case}); (iii) it  can be solved by an analogue of the inverse scattering method (see \cite{AF}); (iv) it
admits  (multi)soliton solutions (see \cite{case}).

In this paper, we consider the BO equation posed on the periodic domain $\T =\R /(2\pi\Z)$:
\begin{equation}
\label{BOT}
u_t + {\mathcal H} u_{xx} + uu_x=0,\quad x\in\T,\ t\in\R,
\end{equation}
where the Hilbert transform $\mathcal H$ is defined now by
\[
(\widehat{{\mathcal H}u})_k=-i\, \text{sgn}(k)\hat u_k.
\]
The two first conserved quantities are
\[
I_1(t)=\int_{\T} u(x,t)dx
\]
and
\[
I_2(t)= \int_{\T} u^2(x,t)dx.
\]
From the historical origins \cite{benjamin,ono} of the BO equation, involving the behavior of stratified fluids, it is
natural to think $I_1$ and $I_2$ as expressing conservation of volume (or mass) and energy, respectively.

The Cauchy problem for the equation \eqref{BOT} in the real line has been intensively studied for many years
(\cite{Sa, Io1,ABFS,Po, MoSaTz, KoTz1, KeKo, tao, BP, IoKe, MP, FoPo, GFFLGP}).
In the periodic case, there have been several recent developments. (See for instance \cite{molinet, MR, MP} and the references therein.)
The best known result so far \cite{molinet,MP} is that the Cauchy problem is well-posed in the space
\[
H^s_0(\T ) =\{u\in H^s(\T ); \  \hat u_0:= \frac{1}{2\pi}\int_{\T}u(x)\, dx=0\}
\]
for $s\ge 0$. Moreover, the corresponding solution map ($u_0\to u$) is real analytic from the space $H^0_0(\T )$ to the space $C([0,T],H^0_0(\T ))$.

In this paper we will study the equation \eqref{BOT} from a control point of view with a forcing term $f=f(x,t)$ added to the equation as a
control input:
\begin{equation}
\label{BOc}
u_t + {\mathcal H} u_{xx} + uu_x=f(x,t),\quad x\in\T,\ t\in\R,
\end{equation}
where $f$ is assumed to be supported in a given open set $\omega \subset \T$. The following exact control problem and stabilization problem
are fundamental in control theory.

{\bf Exact Control Problem:} Given an initial state $u_0$ and a terminal state $u_1$ in a certain space, can one find an appropriate
control input $f$ so that the equation \eqref{BOc} admits a solution $u$ which satisfies $u(\cdot,0)=u_0$ and $u(\cdot,T)=u_1$?

{\bf Stabilization Problem:} Can one find a feedback law $f=Ku$ so that the resulting closed-loop system
\[
u_t + {\mathcal H} u_{xx} + uu_x=Ku,\quad x\in\T,\ t\in\R ^+
\]
is asymptotically stable as $t\to +\infty$?

Those questions were first investigated by Russell and Zhang in \cite{RZ} for the Korteweg-de Vries equation, which serves
as a model for propagation of surface waves along a channel:
\begin{equation}
u_t + u_{xxx} + uu_x=f,\quad x\in\T, \ t\in\R .
\label{BOf}
\end{equation}
In their work, in order to keep
the {\em mass} $I_1(t)$ conserved, the control input  is chosen to be of the form
\[
f(x,t)=(Gh)(x,t):=a(x)\left( h(x,t)-\int_{\T} a(y)h(y,t)\, dy \right)
\]
where $h$ is considered as a new control input, and $a(x)$ is a given nonnegative smooth function such that
$\{x\in \T; \  a(x) >0\}=\omega$ and
\[
2\pi [a]=\int_{\T} a(x)\, dx=1.
\]
For the chosen $a$, it is easy to see that
\[
\frac{d}{dt}\int_{\T} u(x,t)\, dx=\int_{\T} f(x,t)dx=0\quad \forall t\in\R
\]
for any solution $u=u(x,t)$ of the system
\begin{equation}
\label{BOG}
u_t + u_{xxx} + uu_x=Gh,\quad x\in\T,\ t\in\R .
\end{equation}
Thus the {\em mass} of the system is indeed conserved.

The control of dispersive nonlinear waves equations on a periodic domain has been extensively studied
in the last decade: see e.g. \cite{RZ,RZ2009bis,LRZ} for the Korteweg-de Vries equation, \cite{MORZ} for the Boussinesq system,
\cite{RZ2012} for the BBM equation, and
\cite{DGL,RZ2007b,laurent,RZ2009,laurent2} for the nonlinear Schr\"odinger equation.
By contrast, the control theory of the BO equation is at its early stage. The following results are due to Linares and Ortega \cite{LO}.\\[5mm]
{\bf Theorem A.}
\cite{LO}
{\em Let $s\ge 0$ and $T>0$ be given. Then for any $u_0,u_1\in H^s(\T )$ with $[u_0]=[u_1]$ one can find a control input $h\in L^2(0,T,H^s(\T))$ such that
the solution of the system
\begin{equation}
\label{BOlin}
u_t + {\mathcal H} u_{xx} =Gh, \qquad u(x,0)=u_0(x)
\end{equation}
satisfies $u(x,T)=u_1(x)$.}

In order to stabilize \eqref{BOlin}, Linares and Ortega  employed a simple control law
\[
h(x,t)=-G^* u(x,t).
\]
The resulting closed-loop system reads
\[
u_t+{\mathcal H} u_{xx}=-GG^*u.
\]
{\bf Theorem B.} \cite{LO}
{\em Let $s\ge 0$ be given. Then there exist some constants $C>0$ and $\lambda >0$ such that for any $u_0\in H^s(\T )$, the solution of
\[
u_t+{\mathcal H} u_{xx}=-GG^*u,\qquad u(x,0)=u_0(x)
\]
satisfies
\[
\|u(\cdot,t) - [u_0]\|  _{H^s(\T ) }  \le Ce^{-\lambda t} \|u_0-[u_0]\| _{H^s(\T ) } \qquad \forall t\ge 0.
\]}

The extension of those results to the full BO equation \eqref{BOG} turns out to be a very hard task. Indeed, it is by now well known that
the contraction principle cannot be used to establish the local well-posedness of BO in $H^s_0(\T )$ for $s\ge 0$.
The method of proof in \cite{molinet,MP} used strongly Tao's gauge transform, and it is not clear whether this approach
can be followed when an additional  control term is present in the equation.
%It means that the control theory
%for the BO equation cannot be obtained by a classical linearization procedure. Rather, we have to transform the equation by using
%Tao's gauge transform...

For the sake of simplicity, we shall assume from now on that $[u_0]=0$, so that $u(t)$ has a zero mean value for all times.
%Let us consider the BO equation with a forcing term
%\begin{equation}
%\label{BOf}
%u_t + {\mathcal H}u_{xx} + uu_x=f, \qquad u(x,0)=u_0(x).
%\end{equation}
%The first main result in this paper is the following one.
%\begin{theorem}
%\label{main1}
%Let $s\ge 0$ and $T>0$ be given. Then for any $u_0\in H^s_0(\T )$ and any $f\in L^2(0,T,H^s_0(\T))$ there exists a unique solution
%$u$ of \eqref{BOf} in the class...
%Moreover the map $(u_0,f)\to u$ from $H^s_0(\T )\times L^2(0,T,H^s_0(\T)) $ to $C([0,T],H^s_0(\T))$ is of class $C^1$.
%\end{theorem}
%With Theorem \ref{main1} at hand, one can derive a local exact controllability result for the BO equation.

To stabilize the BO equation, we consider the following feedback law
\[
f=-G(D(G u))
\]
where $\widehat{D u}_k =|k| \hat u_k$. Scaling in \eqref{BOf} by $u$ gives (at least formally)
\begin{equation}
\label{id}
\frac{1}{2} \|u(T)\|_{L^2(\T )}^2  + \int_0^T \|D^{\frac{1}{2}}(Gu)\|_{L^2(\T)}^2 dt
=\frac{1}{2} \|u_0\|_{L^2(\T )}^2.
\end{equation}
This suggests that the energy is dissipated over time. On the other hand,
\eqref{id} reveals a smoothing effect, at least in the region $\{ a>0\}$.
Using a {\em propagation of regularity property} in the same vein as  in \cite{DGL,laurent,laurent2,LRZ},
we shall prove that the smoothing effect holds everywhere, i.e.
\begin{equation}
\label{sm}
\|u\|_{L^2(0,T;H^{\frac{1}{2}}(\T ))}\le C(T,\|u_0\|).
\end{equation}
Using this smoothing effect and the classical compactness/uniqueness argument,
we shall first prove that the corresponding closed-loop equation is
semi-globally  exponentially stable.
\begin{theorem}
\label{main1}
Let $R>0$ be given. Then there exist some constants
$C=C(R)$ and $\lambda =\lambda (R)$ such that  for any $u_0\in H^0_0(\T )$ with $\|u_0\|\le R$,
the weak solutions in the sense of vanishing viscosity of
\begin{equation}
\label{BOfinal}
u_t + {\mathcal H} u_{xx} + uu_x =-GDGu,\qquad u(x,0)=u_0(x)
\end{equation}
satisfy
\[
\|u(t)\|\le Ce^{-\lambda t}\|u_0\| \qquad \forall t\ge 0.
\]
\end{theorem}
A weak solution of \eqref{BOfinal} in the sense of vanishing viscosity is a  distributional solution of \eqref{BOfinal}
$u\in C_w(\R ^+ , H^0_0(\T )) \cap L^2_{loc}  ( \R ^+, H_0^{\frac{1}{2}} (\T ))$ that may be
obtained as a weak limit in a certain space of solutions of the BO equation with viscosity
\begin{equation}
\label{BOviscosity}
u_t + ({\mathcal H} -\varepsilon )u_{xx} + uu_x =-GDGu,\qquad u(x,0)=u_0(x)
\end{equation}
as $\varepsilon\to 0^+$ (see below Definition \ref{defiweak} for a precise definition).
The issue of the {\em uniqueness} of the weak solutions in the sense of vanishing viscosity seems challenging.

Using again the smoothing effect \eqref{sm}, one can extend (at least locally) the exponential stability from $H^0_0(\T )$ to $H^s_0(\T )$ for $s>1/2$.
\begin{theorem}
\label{main2}
Let $s\in (\frac{1}{2}, 2]$. Then there exists $\rho >0$ such that for any $u_0\in H^s_0(\T )$  with $\|u_0\|_{H^s(\T )}< \rho$, there exists for all $T>0$ a unique
solution $u(t)$ of \eqref{BOfinal} in the class $C([0,T],H^s_0(\T ))\cap L^2(0,T, H^{s+\frac{1}{2}}_0(\T ))$. Furthermore, there exist some constants $C>0$ and
$\lambda >0$
such that
\[
\|u(t)\| _s \le Ce^{-\lambda t} \|u_0\|_s\qquad \forall t\ge 0.
\]
\end{theorem}

Finally, incorporating the same feedback law $ f =  -G(D(Gu))$ in the control input to obtain a smoothing effect, one can derive
an exact controllability result for the full equation as well.
\begin{theorem}
\label{main3}
Let $s\in (\frac{1}{2}, 2] $ and $T>0$ be given. Then there exists $\delta>0$ such that for any $u_0,u_1\in H^s_0(\T)$ satisfying
\[
\|u_0\|_{H^s(\T)} \le \delta,\quad \|u_1\|_{H^s(\T )}\le \delta
\]
one can find a control input $h\in L^2(0,T,H^{s-\frac{1}{2}} (\T))$ such that the system \eqref{BOG} admits a solution
$u\in C([0,T],H^s_0(\T))\cap L^2(0,T,H^{s+\frac{1}{2}}_0(\T ) )$ satisfying
\[
u(x,0)=u_0(x),\quad u(x,T)=u_1(x).
\]
\end{theorem}
Note that it would be desirable to have a control input $h$ in the class $L^2(0,T,H^s(\T ))$, but this will require to adapt the analysis in
\cite{molinet,MP}. Note also that a global controllability result in $H^0_0(\T )$ would follow from Theorems \ref{main1} and \ref{main3} if
Theorem \ref{main3} were also true for $s=0$.

%As a direct consequence of  Theorems \ref{main1} and \ref{main3}, we obtain the following
%\begin{corollary}
%\label{main4}
%Let $T>0$ and $R>0$ be given. Then for any $u_0,u_1\in H^0_0(\T)$ satisfying
%\[
%||u_0||_{L^2(\T)} \le  R,\quad ||u_1||_{L^2(\T )}\le R
%\]
%one can find a control input $h\in L^2(0,T;H^{-\frac{1}{2}}(\T))$ such that the system \eqref{BOf} with $f=Gh$ admits a solution
%$u\in C([0,T],H^0_0(\T))$ satisfying
%\[
%u(x,0)=u_0(x),\quad u(x,T)=u_1(x).
%\]
%\end{corollary}

The paper is organized as follows. Section 2 is concerned with the local well-posedness  and the stability properties of \eqref{BOfinal}.
We first prove the global well-posedness of \eqref{BOviscosity} in the energy space $H^0_0(\T )$, by using classical energy estimates (Theorem \ref{GWP}).
Next, we establish several technical properties, namely a commutator estimate (Lemma \ref{commutator}), a propagation of regularity
property (Propositions \ref{smoothing} and \ref{prop10}), and a unique continuation property (Proposition \ref{unique_continuation}) that are used
to derive the exponential stability of \eqref{BOviscosity} with a decay rate {\em independent of $\varepsilon$} (Theorem \ref{thmstab1}). This leads to
the proofs of Theorems \ref{main1} and \ref{main2}. Finally,
the control properties of \eqref{BOG} are investigated in Section 3.

\section{Stabilization of BO with a localized damping}
\subsection{Semi-global exponential stabilization in $L^2(\T)$}

\null ~\\
Pick any function
\begin{equation}
\label{defa}
a\in C^\infty (\T ,\R ^+) \ \text{ with }\   \int_{\T }a(x)dx=1
\end{equation}
decomposed as $a(x)=\sum_{k\in\Z}\hat a_ke^{ikx}$.

We  are interested in the stability properties of the BO equation with localized damping
\begin{equation}
\label{BOs}
u_t+ {\mathcal H}u_{xx}+(\frac{u^2}{2})_x=-G(D(Gu)), \qquad u(0)=u_0,
\end{equation}
where
\begin{equation}
\label{defG}
\widehat{{\mathcal H}u}_k=-i\, \text{sgn}(k)\hat u_k,\quad
\widehat{D^s u}_k =|k|^s \hat u_k,\quad
(G u)(x)=a(x)(u(x)-\int_{\T}a(y)u(y)dy).
\end{equation}
We shall assume that $u_0\in H^0_0(\T)$, where for any $s\in\R$,
\[
H^s_0(\T)=\{u=\sum_{k\in \Z} {\hat u}_ke^{ikx}\in H^s(\T );\
{\hat u}_0=0\}.
\]
Let $(u,v)=\int_{\T} u(x)v(x)dx$ denote the usual scalar product in $L^2(\T)$ with
$\|u\|=\|u\|_{L^2(\T)}$ as associated norm, and for any $s\in \R$, let  $(u,v)_s=((1-\partial _x^2)^{\frac{s}{2}} u,(1-\partial _x^2)^{\frac{s}{2}} v)$
denote the scalar product in $H^s(\T)$ with corresponding norm $\|u\|_s=(u,u)_s^{\frac{1}{2} } $. Let
$\langle x\rangle := (1+|x|^2) ^{\frac{1}{2}}$ for any $x\in\R$.

Note that for $s<0$ and $u\in H^s(\T )$, $Gu$ has to be understood as
\[
Gu= a\left( u - \langle u,a\rangle_{H^{s}(\T),H^{-s}(\T)}\right) .
%Gu= a\left( u - (2\pi)^{-1}\sum_{k\in \Z}{\hat u}_k
%\overline{{\hat a}_k}\right)
\]

Assuming that $u_0\in H^0_0(\T)$, we obtain (formally) by scaling in
\eqref{BOs} by $u$  that
\begin{equation}
\label{identity}
\frac{1}{2} \|u(T)\|^2  + \int_0^T \|D^{\frac{1}{2}}(Gu)\|^2 dt
=\frac{1}{2} \|u_0\|^2.
\end{equation}
This suggests that the energy is dissipated over time. On the other hand,
\eqref{identity} reveals a smoothing effect, at least in the region $\{ a>0\}$.
Using a {\em propagation of regularity property} in the same vein as  in \cite{DGL,laurent,laurent2,LRZ},
we shall prove that the smoothing effect holds everywhere, i.e.
\begin{equation}
\label{SE1}
u\in L^2(0,T;H^{\frac{1}{2}}(\T )).
\end{equation}

Of course, a rigorous derivation of \eqref{identity} requires enough regularity for $u$, e.g.
\begin{equation}
\label{SE2}
u\in L^2(0,T,H^1(\T ))\cap C([0,T],H^0_0(\T )).
\end{equation}

As there is a gap between \eqref{SE1} and \eqref{SE2}, we are let to put some artificial viscosity
in \eqref{BOs} (parabolic regularization method) to derive in a rigorous way
the energy identity for the $\varepsilon -$BO equation
\begin{equation}
u_t+{\mathcal H} u_{xx} +uu_x= \varepsilon u_{xx} -G(D(Gu)), \qquad u(0)=u_0.
\label{BOe}
\end{equation}
We shall prove the global well-posedness (GWP) of \eqref{BOe} in $H^0_0$, together with the semi-global exponential stability in $H^0_0$
with a decay rate {\em uniform} in $\varepsilon >0$. Letting $\varepsilon\to 0$, this will
give the semi-global exponential stability in $H^0_0$ of the weak solutions $u\in C_w([0,+\infty ),H^0_0(\T ))$
of \eqref{BOs} obtained as limits of the (strong) solutions of \eqref{BOe}.
The (difficult) issue of the uniqueness of a weak solution to \eqref{BOs}
will not be addressed here.

%Pick any $\varepsilon \in (0,1]$, and consider the regularized equation
%\begin{equation}
%\label{BOreg}
%u_t+ ({\mathcal H}-\varepsilon)u_{xx}+uu_x=-G(D^s(Gu)), \qquad u(0)=u_0,
%\end{equation}
We first establish the GWP of  \eqref{BOe}.
\begin{theorem}
\label{GWP}
Let $\varepsilon >0$ and $u_0\in H^0_0(\T )$. Then for any $T>0$ there exists
a unique solution $u\in C([0,T], H^0_0(\T ))\cap L^2(0,T;H^1(\T))$
of \eqref{BOe}. Moreover
\begin{equation}
\label{parabolic}
u\in C((0,T],H^2(\T ))\cap C^1((0,T],H^1(\T)),
\end{equation}
and for any $t\ge 0$
\begin{equation}
\frac{1}{2}  \|u(t)\|^2
+ \varepsilon \int_0^t \|u_x(\tau )\|^2 d\tau
 + \int_0^t  \|D^{\frac{1}{2}}(G u )  (\tau )  \|^2 d\tau
=  \frac{1}{2}   \|u _0\|^2. \label{identitybis}
\end{equation}
\end{theorem}
\noindent
{\bf Proof:} The proof of Theorem \ref{GWP} is divided into five parts.
%The first part is concerned with the linear theory (semigroup theory and linear estimates),
%while the second part establishes the LWP of the nonlinear equation via the contraction
%principle.
Note that the weak smoothing effect \eqref{SE1} will be established later, as it
is not needed here. \\

\noindent
{\sc Step 1. Linear Theory}\\
We consider the linear system
\[
u_t + ({\mathcal H}-\varepsilon) u_{xx} + G(D(Gu))=0,\qquad u(0)=u_0.
\]
Let $ A u =({\mathcal H} -\varepsilon) u_{xx}$
with domain ${\mathcal D}( A )  = H^2_0(\T )\subset
H^0_0(\T )$, and $ B u = G(D (Gu))$.
Clearly $G\in {\mathcal L}(H^r(\T ),H^r_0(\T ))$ for all $r\in \R$, hence
$ B \in {\mathcal L} ( H^1_0(\T ), H^0_0(\T ))$.
Let $\theta _0\in (\arctan \varepsilon ^{-1}, \pi/2)$. Then, for
$\theta _0 < |\text{arg }\lambda | \le \pi$, we have
\[
\|( A -\lambda )^{-1}\|
\le \sup_{k\ne 0} | (\varepsilon + i \, \text{sgn } k ) k^2 -\lambda |^{-1}
\le \frac{C}{|\lambda |}\cdot
\]
It follows that $A$ is a sectorial operator (see \cite[Definition 1.3.1]{henry}) in $H^0_0(\T )$.
Note that $\sigma ( A )= \{
(\varepsilon + i\, \text{sgn } k)k^2;\ k\in \Z ^* \}$. Therefore,
$\text{Re } \sigma ( A ) \ge \varepsilon$ and
$A^{-\alpha}$ is meaningful for all $\alpha >0$. Since for all $s>0$
\[
\|A^{-\frac{s}{2}} u\|^2_{H^s(\T )}
\le C\sum_{k\ne 0} |\varepsilon + i\, \text{sgn }k|^{-s} |{\hat u}_k|^2
\le C \|u\|^2_{L^2(\T)}
\]
we infer that $ B A^{-\frac{1}{2}} \in {\mathcal L} (H^0_0(\T))$.
It follows from \cite[Corollary 1.4.5]{henry} that the operator
${\mathcal A} := A + B$ is also sectorial, so that $-{\mathcal A}$ generates an
analytic semigroup $\big( {\mathcal S}(t) \big) _{t\ge 0} = (e^{-t {\mathcal A}} )_{t\ge 0}$
on $H^0_0(\T)$ according to  \cite[Theorem 1.3.4]{henry}.
Note that, by \cite[Theorem 1.4.8]{henry}, $D((A+B+\lambda )^\alpha)=D(A^\alpha )=H_0^{2\alpha} (\T ) $ for all
$\alpha\ge 0$ and $\lambda >0$ large enough, hence
\[
{\mathcal S}(t)H_0^{s}(\T )  \subset H_0^{s} (\T ),\qquad \forall t>0,\ \forall s \ge 0.
\]
Let us derive estimates for the solutions of the Cauchy problem
\begin{equation}
\label{AB}
u_t +  {\mathcal A}  u =f,\qquad
u(0)=u_0.
\end{equation}
For any $T>0$ and any $s\in \N$, let
\begin{equation}
Y_{s,T}=C([0,T];H^s_0( \T )) \cap L^2(0,T; H^{s+1}_0(\T ))
\end{equation}
be endowed with the norm
\begin{eqnarray}
\|u\|_{Y_{s,T}} = \|u\|_{L^\infty (0,T;H^s(\T ) )} +\|u\|_{L^2(0,T;H^{s+1}(\T ) )} .
\end{eqnarray}
\begin{lemma}
\label{linearestim}
We have for some constant $C_0=C_0(\varepsilon, s,T)$
\[
\|u\|_{Y_{s,T}} \le C_0 \left( \|u_0\|_{s} + \| f \|_{L^1(0,T,H^s (\T ) ) } \right),
\]
$u$ denoting the mild solution of \eqref{AB} associated with $(u_0,f)\in H^s_0(\T )\times L^1(0,T,H^s_0(\T ))$.
\end{lemma}
\noindent
{\em Proof of Lemma  \ref{linearestim}.}  It is well known from classical semigroup theory that
\[
\|u\|_{L^\infty (0,T,H^s(\T ) ) } \le C\left( \|u_0\|_s + \| f \|_{L^1(0,T,H^s(\T ) ) }  \right) \cdot
\]
Next we estimate $\|u\|_{L^2(0,T,H^{s+1}(\T ))}$. We first
assume $u_0\in H^{s+2}_0(\T )$ and $f\in C([0,T]; H^{s+2}_0(\T ))$, so that
$u\in C([0,T];H^{s+2}_0(\T ))\cap C^1([0,T];H^s_0(\T ))$.
Taking the scalar product of each term of \eqref{AB} by $u$ in $H^s(\T )$
results in
\begin{equation}
\label{energyeps}
\frac{1}{2} \|u(t)\|_{s}^2
+\varepsilon \int_0^t \|u_x\|_s^2\, d\tau
+\int_0^t (G(D(Gu)), u)_s\,  d\tau
 = \frac{1}{2}\|u_0\|_s^2 +
\int_0^t (f,u)_s \, d\tau.
\end{equation}
The identity
\eqref{energyeps} is also true for $u_0\in H^s_0(\T)$ and
$f\in L^1(0,T,H^s_0(\T))$, by density.
The following claim is needed.\\
{\sc Claim 1.}  For any $s\in \R$, there exists a constant $C=C(s)>0$ such that
\[
- \big( G(D(Gu) ) , u\big) _s \le C \|u\|^2_{s} -\|D^{\frac{1}{2}} (Gu)\|^2_s\qquad \forall u\in H_0^{s+1}(\T ).
\]
{\em Proof of Claim 1.} We have
\begin{eqnarray*}
\big( G(D(Gu) ) , u \big)_{s}  &=& \big( (1-\partial^2 _x)^{\frac{s}{2}} G(D(Gu) ) ,
(1-\partial _x^2 )^{\frac{s}{2}} u \big) \\
&=& \big( [(1-\partial ^2_x)^{\frac{s}{2}},G] D(Gu), (1-\partial _x^2)^{\frac{s}{2}} u  )  \\
&&\qquad +(G(1-\partial _x^2)^{\frac{s}{2}} D(Gu), (1-\partial _x ^2)^{\frac{s}{2}} u  ) \\
&=:& I_1+ I_2.
\end{eqnarray*}
Since $a\in C^\infty (\T )$, we easily obtain that
\[
\| [ (1-\partial_x^2)^{\frac{s}{2}},G] u \| \le C \|u\| _{s-1}.
\]
It follows that
\[
| I_1 | \le C \|u\|^2_{s}.
\]
On the other hand
\begin{eqnarray*}
I_2 &=& \big(  (1-\partial _x^2) ^{\frac{s}{2}} D(Gu) ,G(1-\partial _x^2)^{\frac{s}{2}} u \big) \\
&=& \|(1-\partial _x^2 )^{\frac{s}{2}} D^{\frac{1}{2}}(Gu)\| ^2
+ ((1-\partial _x^2)^{\frac{s}{2}}(Gu), D[G,(1-\partial _x^2)^{\frac{s}{2}}]u),
\end{eqnarray*}
hence
\[
-I_2 \le C \|u\|^2_{s} -\|D^{\frac{1}{2}} (Gu)\|^2_s \cdot
\]
The claim is proved.\\
Combining Claim 1 with \eqref{energyeps}, we obtain that  for  $t=T$
\begin{eqnarray*}
&&\frac{1}{2}  \|u(T)\|_s^2
+ \varepsilon \int_0^T \|u_x(\tau )\|_s^2 d\tau
 + \int_0^T  \|D^{\frac{1}{2}}(G u) \|^2_s d\tau \\
 &&\qquad \le \frac{1}{2} \|u_0\|^2_s + C\|u\|^2_{L^2(0,T,H^s(\T ) )} +
 \frac{1}{2}\|u\|^2_{L^\infty (0,T ,H^s(\T ) )}
+ \frac{1}{2}\| f \|^2_{L^1(0,T,H^s( \T )) } \\
&&\qquad  \le C\big( \|u_0\|^2_s + \| f \|^2_{L^1(0,T,H^s(\T ))} \big) \cdot
\end{eqnarray*}
The proof of  Lemma \ref{linearestim} is achieved.\qed

\begin{remark}\label{rem} We observe that when $u_0\equiv 0$ in \eqref{AB} then
\begin{equation}
\|\int_0^t {\mathcal S}(t-\tau )f(\tau )\,d\tau \|_{Y_{s,T}}\le C(\epsilon,s,T)\,\|f\|_{L^1(0,T,H^s(\mathbb T))},
\end{equation}
and when $f\equiv 0$ in \eqref{AB}
\begin{equation}
\|{\mathcal S}(t)u_0\|_{Y_{s,T}}\le C(\epsilon,s,T)\|u_0\|_{H^s(\mathbb T)}.
\end{equation}
\end{remark}

\noindent{\sc Step 2. Local Well-posedness in $H^s_0(\T ), \ s\ge 0$}\\
We prove the following
\begin{proposition}
\label{prop1}
Let $s\ge 0$. For any $u_0\in H^s_0(\T )$, there exists some  $T>0$ such that
the problem \eqref{BOe} admits a unique solution $u\in Y_{s,T}$.
\end{proposition}
\noindent
{\em Proof. \ }
Write \eqref{BOe} in its integral form
\[
u(t)={\mathcal S} (t) u_0-\int_0^t {\mathcal S}(t-\tau ) (uu_x )(\tau ) d\tau
\]
where the spatial variable is suppressed throughout.
For given $u_0\in H^s_0(\T ) $, let $r>0$ and $T>0$ be constants to be determined. Define a map $\Gamma$ on the closed ball
\[
B=\left\{ v\in Y_{s,T};\ \|v\|_{Y_{s,T}} \le r \right\}
\]
of $Y_{s,T}$ by
\[
\Gamma  (v) (t) ={\mathcal S} (t)u_0 - \int_0^t {\mathcal S}(t-\tau )(vv_x)(\tau )\, d\tau .
\]
We aim to prove that $\Gamma$ contracts in $B$
for $T$ small enough and $r$ conveniently chosen. To that end, we shall prove the following estimates
\begin{eqnarray}
\|\Gamma (v)\|_{Y_{s,T}} &\le& C_0\|u_0\|_s + C_1T^{\frac{1}{4}} \|v\|^2_{Y_{s,T}} ,\quad \forall v\in B,\label{est1}\\
\|\Gamma (v^1)-\Gamma (v^2)\|_{Y_{s,T} } &\le& C_1 T^{\frac{1}{4}} (\|v^1\|_{Y_{s,T}}
 + \|v^2\|_{Y_{s,T}}) \|v^1-v^2\|_{Y_{s,T}}  \quad \forall v^1,v^2\in B.\qquad   \label{est2}
\end{eqnarray}

From Lemma \ref{linearestim} and Remark \ref{rem}, it is adduced that
\begin{eqnarray*}
\|\Gamma (v^1) -\Gamma (v^2)\|_{Y_{s,T}}
&\le& C\| v^1 v^1_x   - v^2  v^2_x \|_{L^1(0,T,H^s(\T )) } \\
&\le& C\int_0^T \big( \|v^1-v^2\|_{L^\infty}\|v^1+v^2\|_{s+1} + \|v^1 + v^2\|_{L^\infty} \|v^1-v^2\|_{s+1}\big) d\tau \\
&\le&CT^{\frac{1}{4}} \|v^1-v^2\|_{Y_{s,T}} \big( \|v^1\|_{Y_{s,T}} + \|v^2\|_{Y_{s,T}} \big)
\end{eqnarray*}
where we used the fact that
\[
\int_0^T \|v\|_{L^\infty}^2 dt \le C\int_0^T \|v\|_1  \|v\| dt \le
C\sqrt{T} \|v\|_{L^\infty (0,T,L^2(\T ))} \|v\|_{L^2(0,T,H^1(\T )) }.
\]
This yields \eqref{est2}. \eqref{est1} follows from Lemma \ref{linearestim}, Remark \ref{rem} and \eqref{est2}.
Choosing $r>0$ and $T>0$ so that
\begin{equation}
\left\{
\begin{array}{l}
r=2C_0\|u_0\|_s,\\
2rC_1T^{\frac{1}{4}}\le \frac{1}{2},
\end{array}
\right.
\end{equation}
we obtain that
\[
\|\Gamma (v^1)\| _{Y_{s,T}} \le r, \qquad \|\Gamma (v^1) - \Gamma (v^2)\|_{Y_{s,T}} \le \frac{1}{2} \|v^1-v^2\|_{Y_{s,T}}
\]
for any $v^1,v^2\in B$. Thus, with this choice of $r$ and $T$, $\Gamma$ is a contraction in $B$. Its fixed-point is the unique
solution of \eqref{BOe} in $B$.\\

\noindent{\sc Step 3. Global Well-Posedness in $H^0_0(\T  )$.}\\
Assume that $u_0\in H^0_0(\T )$. We first establish \eqref{identitybis} for $0\le t\le T$.
Since $u\in Y_{0,t}$, we have that
\begin{eqnarray*}
\int_0^t \|uu_x\|^2_{-1} d\tau
&\le&  C \int_0^t \|u^2\|^2 d\tau \\
&\le& C \int_0^t \|u\|^3\|u_x\|\, d\tau \\
&\le& C \sqrt{t} \|u\|^4_{Y_{0,t}}.
\end{eqnarray*}
Thus each term in \eqref{BOe} belongs to $L^2(0,t,H^{-1}(\T ))$. Scaling
in \eqref{BOe} by $u$ yields
\[
\int_0^t\langle u_t+({\mathcal H}-\varepsilon) u_{xx}  +uu_x +G(D(Gu)), u\rangle _{H^{-1}(\T ),H^1(\T ) } d\tau  =0.
\]
We have that for a.e. $\tau\in (0,t)$
\begin{eqnarray*}
&&\langle  ({\mathcal H}-\varepsilon)u_{xx},u\rangle _{H^{-1}(\T ),H^1(\T ) }= - ( ({\mathcal H } - \varepsilon )u_x,u_x)=\varepsilon \|u_x\|^2,\\
&&\langle uu_x,u\rangle _{H^{-1}(\T ), H^1(\T )} = (uu_x,u)=0,\\
&&\langle G(D(Gu)),u\rangle _{H^{-1}(\T ), H^1(\T )} =(G(D(Gu)),u) = \|D^{\frac{1}{2}}(Gu)\|^2.
\end{eqnarray*}
\eqref{identitybis} follows at once, and we infer that $\|u(t)\| \le \|u_0\|$. Using the standard extension argument, one sees that
$u$ is defined on $\R^+$ with $u\in Y_{0,T}$ for all $T>0$. Furthermore, with the constants $C_0$ and $C_1$ given in Step 2 for $s=0$ and
$T=(8C_0C_1\|u_0\|)^{-4}$, we obtain
\[
\|u(nT+\cdot )\|_{Y_{0,T}}\le 2C_0\|u(nT)\|\le 2C_0\|u_0\|.
\]

\noindent{\sc Step 4. Global Well-Posedness in $H^2_0(\T  )$.}\\
Pick any $u_0\in H^2_0(\T )$. By Proposition \ref{prop1} and Step 3, \eqref{BOe} admits a unique solution
$u\in Y_{0,T}$ for each $T>0$, which belongs to $Y_{2,T_0}$ for some $T_0>0$. We just need to show
that $T_0$ may be taken as large as desired. Let $v=u_t$. If $u\in Y_{2,T}$, then $v\in Y_{0,T}$ and it satisfies
\begin{equation}
\label{eqv}
v_t + ( {\mathcal H} -  \varepsilon )v_{xx} + (uv)_x =-G(D(Gv)),\qquad v(0)=v_0
\end{equation}
where
\[
v_0:=-\left\{  ({\mathcal H}-\varepsilon)u_{0,xx} +u_0u_{0,x} + G(D(G u_0)) \right\} \in H^0_0(\T ).
\]
We may write \eqref{eqv} in its integral form
\[
v(t) = {\mathcal S } (t) v_0 - \int_0^t {\mathcal S} (t-s) (uv)_x(s) ds.
\]
Let $\Gamma (w) (t)={\mathcal S } (t) v_0 - \int_0^t {\mathcal S} (t-s) (uw)_x(s) ds$ for $w\in Y_{0,T}$. Computations similar to those
in Step 2 lead to
\begin{eqnarray*}
\|\Gamma w\|_{Y_{0,T}} &\le & C_0\|v_0\| + C_1 T^{\frac{1}{4}} \|u\|_{Y_{0,T}}\|w\|_{Y_{0,T}},\\
\|\Gamma (w^1) -\Gamma (w^2) \|_{Y_{0,T}} &\le& C_1 T^{\frac{1}{4}} \|u\|_{Y_{0,T}} \|w^1-w^2\|_{Y_{0,T}}
\end{eqnarray*}
where the constants $C_0$ and $C_1$ depend only on $\varepsilon$ for $T<1$.  Therefore $\Gamma$ contracts in
$B= \{w\in Y_{0,\theta};\ \|w\|_{Y_{0,\theta }} \le r:= 2C_0\|v_0\|\}$, provided that
\[
C_1\theta ^{\frac{1}{4}}\|u\|_{Y_{0,\theta}} \le \frac{1}{2} \cdot
\]
Its fixed point gives the unique solution of the integral equation in $B$.
Pick $\theta$ fulfilling
\[ \theta<\min \{ (8C_0C_1\|u_0\|)^{-4}, 1\} .\]
Then, from Step 2, we have that
\[
\|u(n\theta +\cdot )\|_{Y_{0,\theta}}\le 2C_0 \|u_0\|
\]
for all $n\in \N$ and that $w$ may be extended to $[n\theta, (n+1)\theta ]$ inductively by using the contraction
mapping theorem (replacing $v_0$ by $w(\theta )$, $w(2\theta )$, etc.). Therefore, $w$ is defined on $\R ^+$ and it holds
\begin{equation}
\label{eqw}
\| w ( n\theta + \cdot ) \|_{Y_{0,\theta }} \le 2C_0 \|w(n\theta )\|\le (2C_0)^{n+1} \|v_0\|.
\end{equation}
By uniqueness of the solution of the integral equation, we have that $v(t)=w(t)$ as long as $0<t<T$ and $v\in Y_{0,T}$.
\eqref{eqw} shows that $\|v(t)\|=\|w(t)\|$ is uniformly bounded  on compact sets of $\R^+$, namely
\[
\|v\|_{Y_{0,T}} \le C(T,\|u_0\|) \|v_0\| .
\]

The same is true  for $\|u(t)\|_2$, by \eqref{BOe}.  Indeed, since
\[
\|uu_x\| \le \|u\|_{L^\infty (\T )} \|u_x\|\le \|u\|^{\frac{5}{4}}\|u_{xx}\|^{\frac{3}{4}}
\le C_\delta \|u\|^5 + \delta \|u_{xx}\| ,
\]
we infer from \eqref{BOe} that
\[
\| ( {\mathcal H}-\varepsilon ) u_{xx} (t)\| \le C(T,\|u_0\|) \|u_0\|_2 + C(\|u\| + \|u\|^5) +  \delta \|u_{xx}\|
\]
hence
\[
\|u(t)\|_2  \le C(T, \|u_0\|)  \|u_0\|_2.
\]
Using the standard extension argument, one sees that $u(t)\in H^2_0(\T )$ for
all $t\ge 0$ with $u\in Y_{2,T}$ for all $T > 0$.\\

\noindent
{\sc Step 5. Smoothing effect from $H^0_0(\T )$ to $H^2_0(\T )$.}\\
Pick any $u_0\in H^0_0(\T )$. Then the solution $u$ to \eqref{BOe} belongs to
$Y_{0,1}$. Therefore, for a.e. $t_0\in (0,1)$, $u(t_0)\in H^1_0(\T )$. The solution of \eqref{BOe}
in $Y_{1,T}$ issued from $u(t_0)$ at $t=0$ must coincide with $u(t_0+t)$ in $[0,T]$,
by uniqueness of the solution of \eqref{BOe} in $Y_{0,T}$.    In particular, $u(t_1)\in H^2_0(\T )$
for a.e. $t_1>t_0$. Again by uniqueness we conclude that $u\in C([t_1,+\infty ),H^2_0(\T))$ for a.e.
$t_1>0$, so that
\[
u\in C((0,+\infty ), H^2_0(\T ))\cap C^1((0,+\infty ),H^0_0(\T )).
\]
The proof of Theorem \ref{GWP} is complete.\qed

The following commutator lemma, used several times in the proof of the property of
propagation of regularity,  is a periodic version of a result from \cite{DMP}.

\begin{lemma}
\label{commutator}
Let ${\mathcal N}\subset \Z$ be a set such that for some constant $C>0$
\begin{equation}
\lan n \ran + \lan k \ran \le C\lan n - k \ran ,\qquad
\forall n\not\in{\mathcal N},\ \forall k\in {\mathcal N}.
\label{separation}
\end{equation}
Let $P$ be the projector on the closure
of $\text{Span}\{ e^{ikx};\ k\in \mathcal N \}$ in
$L^2(\T )$,  namely
$$P
(\sum_{k\in \Z}\hat u_k e^{ikx})=\sum_{k\in \mathcal N}\hat u_ke^{ikx}.
$$
Let $a\in C^\infty (\T )$ and let  $p\in \N,\ q\in \N$. Then there exists
some constant $C=C(a,p,q)>0$ such that for all $v\in L^2(\T )$
\begin{equation}
\label{estim_commutator}
\|\partial _x^p [a,P]\partial _x ^q v\| \le C \|v\|.
\end{equation}
\end{lemma}
\begin{remark}
\label{rmk2.6}
Note that condition \eqref{separation} is fulfilled in the following cases:
(i) ${\mathcal N}=\N ^*$; (ii) ${\mathcal N}$ is a finite set, or the complement of a finite set in $\Z$. It follows that
\eqref{estim_commutator}
is true with $P={\mathcal H}=(-i)(P_{\N ^*} - P_{-\N ^*})$.
Note, however, that condition \eqref{separation} and \eqref{estim_commutator}
are not true when $\mathcal N=1+2\Z$ (pick e.g. $a(x)=e^{ix}$).
\end{remark}
\noindent
{\em Proof of Lemma \ref{commutator}.} Let ${\mathcal N},a,p$ and $q$ be as in the statement of the lemma, and pick any
$v\in C^\infty (\T )$. Decompose $a$ and $v$ in using Fourier series
$$
v(x)=\sum_{n\in \Z}{\hat v}_n e^{inx},\qquad a=\sum_{n\in \Z}
{\hat a}_n e^{inx},
$$
and denote by $1_{\mathcal N}$ the characteristic function of
$\mathcal N$, defined by $1_{\mathcal N}(n)=1$ if $n\in \mathcal N$, and $0$
otherwise. Then
\begin{eqnarray*}
[a,P]v &=&
a(Pv) -P(av)\\
&=& a (\sum_{n} 1_{\mathcal N}(n) {\hat v}_n e^{inx}) -
P(\sum_n(\sum_k {\hat a}_{n-k}{\hat v}_k )e^{inx}) \\
&=&  \sum_{n} \left( \sum_k {\hat a}_{n-k}{\hat v}_k
(1_{\mathcal N}(k) -1_{\mathcal N} (n)) \right) e^{inx}.
\end{eqnarray*}
Taking derivatives, one obtains
$$\partial_x^p [a,P]\partial _x^q v =
\sum_n \left( \sum_k {\hat a}_{n-k}(ik)^q{\hat v}_k
(1_{\mathcal N}(k) -1_{\mathcal N}(n))\right)(in)^p e^{inx}
 =:\Sigma _1 - \Sigma _2
$$
where $\Sigma _1$ (resp. $\Sigma _2$) is the sum over the $(n,k)$ with
$n\not\in\mathcal N$ and $k\in\mathcal N$
 (resp. with $n\in\mathcal N$ and $k \not\in\mathcal N$).
Let us estimate $\Sigma _1$
only, the estimate for $\Sigma _2$ being similar. Since $a\in C^\infty (\T )$,
for any $s\in \N$ there exists some constant $C_s>0$ such that
\begin{equation}
\label{smooth}
|{\hat a}_l|\le C_s \langle l\rangle ^{-s}\qquad \forall l\in \Z.
\end{equation}
Then, for $s>\sup\{ 2p+1,2q+1\}$,
\begin{eqnarray*}
\|\Sigma _1 \|^2_{L^2(\T )}
&=& \| \sum_{n \not\in\mathcal N}(\sum_{k\in\mathcal N} {\hat a}_{n-k}{\hat v}_k (ik)^q (in)^p )e^{inx} \|^2_{L^2(\T )} \\
&=& C  \sum_{n\not\in\mathcal N}
\left\vert \sum_{k\in\mathcal N}{\hat a}_{n-k}
{\hat v}_k(ik)^q \right\vert ^2 |n|^{2p}\\
&\le& C \|v\|^2\sum_{n\not\in\mathcal N}\sum_{k\in\mathcal  N}\langle n-k\rangle ^{-2s} |n|^{2p}|k|^{2q}\\
&\le& C \|v\|^2\sum_{n\not\in\mathcal N}\sum_{k\in\mathcal  N}
(\langle n \rangle + \langle k\rangle ) ^{-2s} |n|^{2p}|k|^{2q}\\
&\le& C\|v\|^2
\end{eqnarray*}
where we used the Cauchy-Schwarz inequality, \eqref{smooth} and \eqref{separation}. Since $C^\infty (\T )$ is dense in $L^2(\T)$,
the proof is complete.\qed \\
The propagation of regularity property we need is as follows.
\begin{proposition}
\label{smoothing}
Let $a\in C^\infty (\T, \R ^+)$, $\varepsilon >0$,  $\alpha \in\R$, $T>0$, and $R>0$ be given. Pick any
$v_0\in H^0_0(\T)$ with $\|v_0\| \le R$ and let $v\in C([0,T];H^0_0(\T ))\cap L^2(0,T,H^1(\T ))\cap C((0,T],H^2(\T ))$ be such that
\begin{eqnarray}
&& v_t + ({\mathcal H} -\varepsilon ) v_{xx} + \alpha v v_x =
- G (D(G \,v)),\qquad x\in \T,\ t\in (0,T)\\
&&v(0)=v_0.
\end{eqnarray}
Then there exists some constant $C=C(T)>0$ (independent of $\varepsilon$, $\alpha$ and $R$) such that
\begin{equation}
\label{H12}
\int_0^T \|D^{\frac{1}{2}}v\|^2dt \le C(R^2 + \alpha ^4 R^6).
\end{equation}
\end{proposition}
\noindent
{\em Proof of Proposition \ref{smoothing}}. Pick any $t_0\in (0,T)$. Let $(f,g)_{L^2_{t,x}}:=\int_{t_0}^T \int_{\T} f(x,t)g(x,t)\, dxdt$ denote
the scalar product in $L^2(t_0,T,L^2(\T ))$.
$C$ will denote a constant which may vary from line to line,  and which may depend on $T$, but not  on $t_0$, $\varepsilon$, $\alpha$ and $R$.
Setting $Lv:=v_t+{\mathcal H}v_{xx}$,
$f:=\varepsilon v_{xx} -G(D(Gv))$ and $g:=-\alpha vv_x$, we have that
$$
Lv=f+g.
$$
Pick any $\varphi\in C^\infty (\T )$, and set $Av=\varphi (x)v$. Noticing that
$L$ is formally skew-adjoint, we have that
\begin{eqnarray*}
([L,A]v,v)_{L^2_{t,x}}
&=&(L(\varphi v) -\varphi (Lv),v)_{L^2_{t,x}}\\
&=&(\varphi v, L^*v)_{L^2_{t,x}}  + [(\varphi v,v)]_{t_0}^T -(Lv,\varphi v)_{L^2_{t,x}}
\end{eqnarray*}
so that
\[
|  ([L,A]v,v)_{L^2_{t,x}} | \le 2|(f + g, \varphi v)_{L^2_{t,x}}| + 2  \|\varphi\|_{L^\infty (\T )} R^2 .
\]
We first notice that
\begin{eqnarray*}
|(f,\varphi v)_{L^2_{t,x}}|
&\le& |(v_x,\varepsilon (\varphi v)_x)_{L^2_{t,x}}|  +
|(D(Gv),G(\varphi v))_{L^2_{t,x}}|\\
&\le&  C \varepsilon \int_0^T \int_{\T}
(|v|^2 + |v_x|^2)dt + C\int_0^T  \|D^{\frac{1}{2}}(Gv)\|^2  dt \\
&&\qquad +\int_0^T \{ |(D(Gv),[G,\varphi ] v)| +|(D^{\frac{1}{2}}(Gv),[D^{\frac{1}{2}},\varphi](Gv))| \} d\tau \\
&\le& CR^2
\end{eqnarray*}
where we used \eqref{identitybis} and classical commutator estimates.
(Note that Theorem \ref{GWP} is still true when $\alpha =1$ is replaced by any value $\alpha \in \R$.)
On the other hand
$$
|(g,\varphi v)_{L^2_{t,x}}|=|(\alpha vv_x,\varphi v)_{L^2_{t,x}}|=\frac{| \alpha |}{3} |(v^3,\varphi _x)_{L^2_{t,x}}|.
$$
From Sobolev embedding and the fact that the $L^2-$norm is nonincreasing
$$
\|v\|_{L^3} \le \|v\|^{\frac{1}{2}}\|v\|^{\frac{1}{2}}_{L^6}
\le C R^{\frac{1}{2}} \|v\|^{\frac{1}{2}}_{\frac{1}{2}} \cdot
$$
Therefore,
\begin{eqnarray*}
|(g, \varphi v)_{L^2_{t,x}} |
&\le& C| \alpha |  \int_{t_0}^T\|v\|^3_{L^3} dt \\
&\le&  C|\alpha | R^{\frac{3}{2}}T^{\frac{1}{4}}
\left( \int_{t_0}^T \|v\|^2_{\frac{1}{2}}dt\right)^{\frac{3}{4}}\\
&\le& C\delta ^{-3}\alpha ^4 R^{6} T + \delta \int_{t_0}^T\| D^{\frac{1}{2}} v\|^2  dt
\end{eqnarray*}
where $\delta >0$ will be chosen later on.
On the other hand
\begin{eqnarray}
[L,A]v &=& [{\mathcal H}\partial _x^2,\varphi]v \nonumber \\
&=& {\mathcal H}\big( (\partial _x^2\varphi )v +2 (\partial _x \varphi)
(\partial _x v)
+\varphi\partial _x^2v \big) -\varphi {\mathcal H}\partial _x^2 v  \nonumber \\
&=& [{\mathcal H},\varphi ] \partial _x ^2 v + {\mathcal H}\big( (\partial^2 _x \varphi )v\big)
+2[{\mathcal H},\partial _x \varphi ]\partial _x v + 2(\partial _x\varphi) {\mathcal H}\partial _x v.
\label{K1}
\end{eqnarray}

It follows from Lemma \ref{commutator}  and Remark \ref{rmk2.6} that
\begin{eqnarray}
&&| ( [{\mathcal H},\varphi ] \partial _x^2 v , v )_{L^2_{t,x}} |
+ |\big(   {\mathcal H}( (\partial _x ^2 \varphi ) v), v \big)_{L^2_{t,x}} |
+ |\big(  [{\mathcal H},\partial _x\varphi ] \partial _x v , v \big) _{L^2_{t,x}} |
\nonumber\\
&&\qquad \le C\|v\|^2_{L^2(0,T;L^2(\T ))}\nonumber \\
&&\qquad \le C R ^2. \label{K2}
\end{eqnarray}
Therefore
$$
|(\partial _x \varphi {\mathcal H}\partial _x v, v)_{L^2_{t,x}}|
\le C( R^2 +\delta ^{-3}\alpha ^4 R^6) + \delta  \int_{t_0}^T \|D^\frac{1}{2} v\|^2 dt.
$$
Let $b\in C_0^\infty(\omega)$, where $\omega = \{ x\in \T; \ a(x)>0 \}$.
Then $b=a{\tilde b }$ with ${\tilde b}\in C_0^\infty (\omega)$
and
\begin{eqnarray}
\int_{t_0}^T \|D^{\frac{1}{2}} (bv)\|^2 dt
&\le& 2\int_{t_0}^T \big(  \| [D^{\frac{1}{2}},\tilde b](av)\|^2 + \|\tilde b D^{\frac{1}{2}}(av)\|^2\big) dt\nonumber \\
&\le&  C\int_{t_0}^T (\|v\|^2 + \|D^{\frac{1}{2}} (av)\|^2) dt \nonumber \\
&\le& C\int_0^T \big( \|v\|^2 + \|D^{\frac{1}{2}}(Gv)\|^2 +\|D^{\frac{1}{2}}a\|^2 |\int_{\T}a(y) v(y,t)\, dy|^2 \big)dt \nonumber\\
&\le&  CR^2.
\label{Z1}
\end{eqnarray}
Pick any $x_0\in \T$. Then $b^2(x)-b^2(x-x_0)=\partial _x\varphi $
for some $\varphi\in C^\infty (\T )$. Noticing that ${\mathcal H}\partial _x =D$, we have that
\begin{eqnarray*}
|(b^2(x){\mathcal H}\partial _x v,v)_{L^2_{t,x}}|
&=& |(b D v, b v)_{L^2_{t,x}}| \\
&\le& | ([b,D]v,bv)_{L^2_{t,x}} |  +  |( D(bv),bv)_{L^2_{t,x}}|\\
&\le& C\|v\|^2_{L^2(0,T;L^2(\T ))} +
\int_{t_0}^T \| D^{\frac{1}{2}} (bv)\|^2dt\\
&\le& CR^2
\end{eqnarray*}
by \eqref{Z1}. It follows that
$$|(b^2(x-x_0)Dv,v)_{L^2_{t,x}}|\le C(R^2 +\delta ^{-3}\alpha ^4 R^6)  + \delta \int_{t_0}^T \|D^\frac{1}{2} v\|^2dt.$$
Using a partition of unity and choosing $\delta >0$ small enough, we infer that
$$
|(Dv,v)_{L^2_{t,x}}| \le C(R^2 +\alpha ^4 R^6)  +\frac{1}{2} \int_{t_0}^T \|D^\frac{1}{2} v\|^2 dt.
$$
This gives
$$
\int_{t_0}^T \|D^\frac{1}{2} v\|^2dt \le C(R^2 +\alpha ^4 R^6),
$$
where $C=C(T)$. Letting $t_0\to 0$ yields the result.
\qed\\

A unique continuation property is also required.
\begin{proposition}
\label{unique_continuation}
Let $\alpha\in \R$, $\varepsilon \ge 0$,   $c\in L^2(0,T)$, and $u\in L^2(0,T;H^0_0(\T ) )$ be such that
\begin{eqnarray}
u_t+ ({\mathcal H} -\varepsilon )u_{xx} +\alpha uu_x=0 &&\text{ in }\  \T \times (0,T),\label{UCP1}\\
u(x,t)=c(t) &&\text{ for a.e. }\ (x,t)\in (a,b)\times (0,T)\label{UCP2}
\end{eqnarray}
for some numbers $T>0$ and $0\le a<b\le 2\pi$. Then $u(x,t)=0$ for a.e.
$(x,t)\in \T\times (0,T)$.
\end{proposition}
\noindent
{\em Proof.} From \eqref{UCP2}, we obtain that
$u_{xx}(x,t)=(uu_x)(x,t)=0$ for a.e. $(x,t)\in (a,b)\times (0,T)$. Thus, by using
\eqref{UCP1},
$$
{\mathcal H} u_{xx} =-u_t= -c_t \qquad \text{ in } (a,b)\times (0,T).
$$
Therefore, for almost every $t\in (0,T)$, it holds
\begin{eqnarray}
&& u_{xxx}(\cdot,t)\in H^{-3}(\T),\\
&& u_{xxx}(\cdot,t)=0 \ \text{ in }\ (a,b),\\
&& {\mathcal H}u_{xxx}(\cdot,t)=0 \ \text{ in }\ (a,b).
\end{eqnarray}
Pick a time $t$ as above, and set $v=u_{xxx}(\cdot,t)$. Decompose $v$ as
$$
v(x)=\sum_{k\in \Z} \hat v_k e^{ikx},
$$
the convergence of the Fourier series being in $H^{-3}(\T )$.
Then in $(a,b)$
$$
0=iv - {\mathcal H}v = 2i\sum_{k > 0}\hat v_k e^{ikx}.
$$
Since $v$ is real-valued, we also have that $\hat v_{-k}=\overline{\hat v_k}$
for all $k$.
The following lemma for Fourier series is needed.
\begin{lemma}
\label{Fourier}
Let $s\in \R$ and let $v(x) =\sum_{k\ge 0}\hat v_k e^{ikx}$ be such that
$v\in H^s(\T )$ and $v=0$ in $(a,b)$. Then $v\equiv 0$.
\end{lemma}
\noindent
{\em Proof of Lemma \ref{Fourier}.} It is clearly sufficient to prove
the property for $s=-p$, where $p\in \N$. Let us proceed by induction on $p$.
Assume first that $p=0$. Then
\begin{equation}
\label{L2}
\sum_{k\ge 0}|\hat v_k|^2 <\infty.
\end{equation}
Introduce the set $U=\{ z\in \C ;\ |z|<1 \}$ and
the Hardy space (see e.g. \cite{rudin})
$$
{\mathbf H}^2 (U) =\{ f : U\to \C;\ f  \text{ is holomorphic in $U$ and }
\limsup_{r\to 1^-} \int_{-\pi}^\pi |f(re^{i\theta})|^2 d\theta <\infty \}
$$
Let $f(z)=\sum_{k\ge 0}\hat v_k z^k$. Then, by \cite[Thm 17.10]{rudin} and
\eqref{L2}, we have that $f\in {\mathbf H}^2(U)$. On the other hand, by
\cite[Thm 17.10 and Thm 17.18]{rudin}, it holds that
\begin{eqnarray}
&&f^* ( e^{i\theta} ): =\lim_{r\to 1^-} f( re^{i\theta} )
\text{ exists for a.e. } \theta \in (0,2\pi); \label{F1}\\
&& f^*(e^{i\theta}) =\sum_{k\ge 0} \hat v_k e^{ik\theta } = v(\theta )
\qquad \text{ in } L^2(\T ); \label{F1bis}\\
&&\text{If } f\not\equiv 0,\ \ \text{then} \ f^*(e^{i\theta })\ne 0 \ \ \text{for a.e. } \theta\in (0,2\pi ).\label{F2}
\end{eqnarray}
%By Lebesgue dominated convergence theorem, we see that
%$$
%\lim_{r\to 1^-} f(re^{i\theta }) = \sum_{k\ge 0} \hat v_k e^{ik\theta}=
%v(\theta )\qquad \text{ in } L^2(\T )
%$$
%so, for a sequence $r_n\nearrow 1^-$,
%$\lim_{n\to \infty} f(r_ne^{i\theta})=v(\theta )$ a.e. This gives
Since
$$
f^*(e^{i\theta })=v(\theta )=0\qquad \text{ for a.e. } \theta \in (a,b),
$$
it follows from \eqref{F2} that $f\equiv 0$. Therefore $\hat v_k=0$ for all
$k\ge 0$, hence $v\equiv 0$. This gives the result for $p=0$.
Assume now that the result has been proved for $s= - p$ for some $p\in \N$, and pick any
$v\in H^{-p-1}(\T )$, decomposed as $v(x)=\sum_{k\ge 0}\hat v_k e^{ikx}$, and such
that  $v\equiv 0$ in $(a,b)$.
Let $w(x)=\sum_{k>0} \frac{\hat v_{k-1}}{ik} e^{ikx}$. Then
$w\in H^{-p}(\T )$ and
$$
w_x=\sum_{k>0} \hat v_{k-1}e ^{ikx} = e^{ix}v,
$$
so $w_x=0$ on $(a,b)$ and we have, for some constant $C\in\C$,
\begin{equation}
\label{constant}
w(x)=C\ \ \text{  on } \  (a,b).
\end{equation}
 Introducing
the function $\tilde w(x)=w(x)-C$, we infer from
\eqref{constant} and the induction hypothesis that $\tilde w\equiv 0$ on
$\T$, which yields $v\equiv 0$ on $\T$. This completes the proof of
Lemma \ref{Fourier}.\qed

With Lemma \ref{Fourier} we infer that for a.e. $t\in (0,T)$, $u_{xxx}(.,t)=0$ in $\T$, hence with \eqref{UCP2}
$u(x,t)=c(t)$ a.e. in $\T\times (0,T)$. From \eqref{UCP1} we infer that $c_t=0$, which, combined with the fact that  $u\in L^2(0,T;H^0_0(\T))$,
gives that $u(x,t)=0$ a.e. in $\T\times (0,T)$. The proof of Proposition \ref{unique_continuation} is complete.\qed

We are now in a position to state a stabilization result for the $\varepsilon$-BO equation. We stress that the
decay rate does not depend on  $\varepsilon$.
\begin{theorem}
\label{thmstab1}
Let $R>0$.
There exist some numbers $\lambda >0$ and $C>0$ such that
for any $\varepsilon \in (0,1]$ and any $u_0\in H^0_0 (\T )$ with $\|u_0\|\le R$, the solution  $u$ of \eqref{BOe} satisfies
$$
\|u(t)\| \le Ce^{-\lambda t}\|u_0\|\qquad \forall t\ge 0.
$$
\end{theorem}
\noindent {\em Proof.}  Note that $\|u(t)\|$ is nonincreasing by \eqref{identitybis}, so that
the exponential decay is ensured if $\|u((n+1)T)\|\le \kappa \|u(nT)\|$ for some $\kappa <1$.
To prove the theorem, it is thus sufficient (with \eqref{identitybis}) to establish the following observability
inequality: for any $T>0$ and any $R>0$ there exists some constant $C(T,R) >0$ such that for
any $\varepsilon \in (0,1]$ and any $u_0\in H^0_0(\T ) $ with $\|u_0\|\le R$, it holds
\begin{equation}
\label{obs}
\|u_0\|^2 \le C\left( \varepsilon   \!\! \int_0^T \|u_x(t )\|^2 dt + \int_0^T \|D^{\frac{1}{2}}(Gu)\|^2 dt\right) ,
\end{equation}
where $u$ denotes the solution of \eqref{BOe}. Fix any $T>0$ and any $R>0$, and assume
that \eqref{obs} fails. Then there exist a sequence $ (u_0^n )$ in $H^0_0(\T )$ and a sequence $ (\varepsilon ^n )$ in $(0,1]$ such
that for each $n$ we have $\|u_0^n\|\le R$, and
$$
\|u^n_0\|^2 > n  \left( \varepsilon ^n \!\! \int_0^T \|u^n_x (t)\|^2 \, dt +  \int_0^T \|D^{\frac{1}{2}}(G \,u^n)\|^2dt\right) .
$$
Let $\alpha ^n=\|u_0^n\|\in (0,R]$. Extracting a sequence if needed, we may assume that  $\alpha ^n\to \alpha \in [0,R]$ and
$\varepsilon ^n\to \varepsilon \in [0,1]$. Let $v^n=u^n/\alpha ^n$.  Then $v^n$ solves
\begin{equation}
\label{BO2}
v^n_t+ ({\mathcal H}-\varepsilon ^n)v^n_{xx} + \alpha ^n v^n v^n_x = - G (D ( G v^n)), \qquad v^n(0)=v^n_0
\end{equation}
with $v^n_0\in H^0_0(\T ) $ and $\|v^n_0\|=1$.
Again, we have that
\begin{eqnarray}
&&\frac{1}{2} \|v^n(t)\|^2 + \varepsilon ^n \int_0^t \|v^n_x \|^2 d\tau +  \int_0^t \|D^{\frac{1}{2}}(G v^n)\|^2 d\tau
=\frac{1}{2} \|v_0^n\|^2\qquad \forall t>0,\label{X1}\\
&&1=\|v_0^n\|^2 > n
\left( \varepsilon ^n \int_0^T \|v^n_x(t)\|^2 dt +  \int_0^T \|D^{\frac{1}{2}}(G v^n)\|^2\, dt \right) .
\label{X2}
\end{eqnarray}
We infer from Proposition \ref{smoothing} that
\begin{equation}
\label{H12bis}
\int_0^T \|D^\frac{1}{2} v^n\|^2dt \le C.
\end{equation}
This yields
$$
\|G( D(Gv^n))\|_{ L^2(0,T; H^{-\frac{1}{2}}   (\T )) } +
\| ( {\mathcal H}-\varepsilon )v^n_{xx}\|_{L^2(0,T;H^{-\frac{3}{2}} (\T )) } \le C.
$$
On the other hand, for any $\delta >0$
$$
\|v^nv^n_x\|_{H^{-\frac{3}{2}-\delta}(\T )}
\le C \|(v^n)^2\|_{ H^{-\frac{1}{2}-\delta } (\T )}
\le C \|(v^n)^2\|_{L^1(\T )}
\le C \|v^n\|^2 \le C
$$
thus
$$
\|\alpha ^n v^n v^n_x\|_{L^2(0,T;H^{-\frac{3}{2}-\delta}(\T ))} \le C.
$$
It follows that $ (v^n_t  )$ is bounded in
$L^2(0,T; H^{-\frac{3}{2} -\delta }(\T ))$. Combined with \eqref{H12bis} and Aubin-Lions'  lemma, this gives that
for a subsequence still denoted by $ (v^n )$, we have %as $n\to \infty$
\begin{eqnarray*}
&&v^n\to v\qquad \text{ in } L^2(0,T;H^\alpha (\T )) \qquad\forall \alpha <\frac{1}{2}, \\
&&v^n\to v\qquad \text{ in } L^2(0,T;H^{\frac{1}{2}} (\T ))\text{ weak} , \ \\
&&v^n\to v\qquad \text{ in } L^\infty (0,T;L^2 (\T )) \text{ weak}* \
\end{eqnarray*}
for some function
$
v\in L^2(0,T;H_0^\frac{1}{2} (\T ))
\cap L^\infty (0,T;L^2(\T ))$.
In particular,
$$
(v^n)^2\to v^2 \qquad \text{ in } L^1(\T \times (0,T)).
$$
Letting $n\to \infty$ in \eqref{X2}, we obtain that
$$
\int_0^T\|D^{\frac{1}{2}} (Gv)\|^2 dt =0,
$$
hence $Gv=0$ a.e. on $\T \times (0,T)$. Recall that  $\omega=\{ x\in \T; \   a(x)>0 \}$. Then
\[
v(x,t)=\int_{\T}a(y)v(y,t)\, dy= : c(t)\qquad \text{ for a.e. } (x,t)\in \omega\times (0,T).
\]
Note that $c\in L^\infty (0,T)$.
Taking the limit in \eqref{BO2} gives
\begin{equation*}
\left\{
\begin{array}{ll}
v_t+({\mathcal H}-\varepsilon ) v_{xx}+\alpha vv_x=0,\qquad  &\text{in } \T \times (0,T),\\
v(x,t) =c(t)\qquad &\text{for a.e. } (x,t) \in \omega  \times (0,T).
\end{array}
\right.
\end{equation*}
It follows from Proposition \ref{unique_continuation} that $v\equiv 0$.
Thus, extracting a subsequence still denoted by $ (v^n )$, we have that $v^n(\cdot ,t)\to 0$ in $L^2(\T)$ for a.e. $t\in (0,T)$.
Using \eqref{X1}-\eqref{X2}, we infer that $v^n_0\to 0$ in $L^2(\T )$. This contradicts the fact that $\|v^n_0\|=1$ for all $n$.
\qed

We are now in a position to define the weak solutions of \eqref{BOs} obtained by the method of
vanishing viscosity, and to state the corresponding exponential stability property.
\begin{definition}
\label{defiweak}
For $u_0\in H^0_0(\T )$, we call a {\em weak solution of \eqref{BOs} in the sense of vanishing viscosity}
any function $u\in C_w(\R ^+, H^0_0(\T ))$ with $u\in L^2(0,T, H^{\frac{1}{2}}(\T ))$ for all $T>0$  which solves
\eqref{BOs} (in the distributional sense) and such that for some sequence $\varepsilon ^n \searrow 0$ we have
for all $T>0$
\begin{eqnarray*}
u^n &\to& u \qquad \text{ in } L^\infty (0,T,H^0_0 (\T ))\ \text{weak}\, *,\\
u^n &\to& u  \qquad \text{ in } L^2(0,T,H^{\frac{1}{2}} _0(\T ))\ \text{weak}
\end{eqnarray*}
where $u^n$ solves \eqref{BOe} for $\varepsilon =\varepsilon ^n$.
\end{definition}
The main result in this section is the following
\begin{theorem}
\label{Thmstab}
For any $u_0\in H^0_0(\T )$ there exists (at least) one weak solution of \eqref{BOs} in the sense of vanishing viscosity. On the other hand, for all
$R>0$ there exist some positive constants $\lambda =\lambda (R)$ and $C=C(R)$ such that for any weak solution $u(t)$ of \eqref{BOs} in the sense of
vanishing viscosity, it holds
\begin{equation}
\label{decay}
\|u(t)\| \le Ce^{-\lambda t} \|u_0\| \qquad \forall t\ge 0
\end{equation}
whenever $\|u_0\|\le R$.
\end{theorem}
\noindent
{\em Proof.} Pick $R>0$  and $u_0\in H^0_0(\T ) $ with $\|u_0\|\le R$. Pick any sequence $\varepsilon ^n\searrow 0$ and let $u^n(t)$ denote
the solution of
\begin{equation}
\label{BOn}
u^n_t + ({\mathcal H}-\varepsilon ^n) u^n_{xx} +u^n u^n_x = -G(DG u^n),\quad u^n(0)=u_0.
\end{equation}
It follows from \eqref{identitybis} and \eqref{H12} that
\begin{eqnarray*}
&&\|u^n\|_{L^\infty (0,T,H^0_0(\T ))} \le R,\\
&&\|u^n\|_{L^2(0,T,H_0^{\frac{1}{2}} ( \T ))} \le C(T,R).
\end{eqnarray*}
Using a diagonal process, we obtain that for a subsequence, still denoted by $( u^n )$, we have for all $T>0$
\begin{eqnarray}
&&u^n\to u \qquad \text{ in } L^\infty(0,T,H^0_0(\T )) \ \text{ weak}*,\label{H99}\\
&&u^n\to u \qquad \text{ in } L^2(0,T,H_0^{\frac{1}{2}} (\T )) \ \text{weak} \label{H100}
\end{eqnarray}
for some function $u\in L^\infty (\R ^+ ,H^0_0(\T ))\cap L^2_{loc}(\R ^+,H^{\frac{1}{2}}_0(\T ))$. The same argument as in the proof of
Theorem \ref{thmstab1} shows that $\{ u^n_t \}$ is bounded in
$L^2(0,T; H^{-\frac{3}{2} -\delta }(\T ))$ for all $\delta >0$. Combined with \eqref{H99}-\eqref{H100} and Aubin-Lions'  lemma, this shows that
\[u^n\to u \ \text{ in }\  L^2(\T \times (0,T)) \text { and in } C([0,T],H_0^{-\delta }(\T ))\]
 for all $T>0$ and all $\delta >0$.  On the other hand, $u\in C([0,T],H^{-\delta}(\T ))$ for all $T>0$ and all $\delta >0$, which, combined to
 \eqref{H99}, yields $u\in C_w(\R ^+,H^0_0(\T ))$ (the space of weakly continuous functions from $\R ^+$ to $H^0_0(\T )$).
 By letting $n\to \infty$ in \eqref{BOn}, we see that $u$ solves \eqref{BOs}. Thus $u$ is a weak solution of \eqref{BOs} in the sense of
 vanishing viscosity. On the other hand, from Theorem \ref{thmstab1} we have that
 \[
 \|u^n(t)\| \le C e^{-\lambda t}\|u_0\| ,\quad \forall t\ge 0, \ \forall n\ge 0.
 \]
 where $C=C(R)$, $\lambda = \lambda (R)$. Letting $n\to \infty$ in the above estimate yields \eqref{decay}. Note also that
$\|u(t)\|\le \|u_0\|$ for all $t\ge 0$, since the same estimate holds for the $u^n$'s and $u\in C_w (\R ^+ ,H^0_0(\T ))$.\qed
\subsection{Local stabilization in $H^s_0(\T )$}
\null ~\\
\subsubsection{Main results}
Let again $a$ and $G$ be as in \eqref{defa} and \eqref{defG}, respectively.
For $s\ge 0$ and $T>0$, let
\begin{equation}
\label{R2}
Z_{s,T}=C([0,T],H^s_0(\T )) \cap L^2(0,T, H_0^{s+\frac{1}{2}} (\T ))
\end{equation}
be endowed with the norm
\[
\| v\|_{Z_{s,T}} =\|v\|_{L^\infty(0,T,H^s(\T ))} + \|v\|_{L^2(0,T,H^{s+\frac{1}{2}}(\T ))} \cdot
\]

We are concerned here with the stability properties of the BO equation
with localized damping \eqref{BOs} in the space $H^s_0(\T )$ for $s>0$.  Our first aim is to prove the local well-posedness of \eqref{BOs} in
$H^s_0(\T )$ for $s>1/2$.
%%%%%%%%%%%%%%%%%%%%%%%%%%%%%%%%%%%%%%%%%%%%%%%%%%%%%%%%%%%%%%%%%%%%%
\begin{theorem}
\label{thmA}
Let $s\in (\frac{1}{2},2]$. Then there exists $\rho >0$ such that for any $u_0\in H^s_0(\T )$ with $\|u_0\|_s <\rho $, there exists some time $T>0$
such that \eqref{BOs} admits a unique solution in the space
$Z_{s,T}$.
\end{theorem}
%%%%%%%%%%%%%%%%%%%%%%%%%%%%%%%%%%%%%%%%%%%%%%%%%%%%%%%%%%%%%%%%%%%%%
The proof of Theorem \ref{thmA} rests on the smoothing effect due to the damping term, namely
\begin{equation}
\label{R3}
\int_0^T \|e^{-t({\mathcal H} \partial _x ^2 +GDG)} u_0\|^2_{\frac{1}{2}} dt \le C\|u_0\|^2.
\end{equation}
In \cite{RZ2006}, the semi-global exponential stability of the Korteweg-de Vries on a bounded domain $(0,L)$
with a localized damping was first established in $L^2(0,L)$, and next extended to $\{u\in H^3(0,L);\ \ u(0)=u(L)=u_x(L)=0\}$ by using the
Kato smoothing effect in the equation fulfilled by the time derivative of the solution. As the smoothing effect \eqref{R3} is much weaker,
that argument cannot be used. The semi-global exponential stability of \eqref{BOs} in $H^0_0(\T )$, if true, is thus open. However, a local
exponential stability in $H^s_0(\T )$ for $s>1/2$ can be derived.
%%%%%%%%%%%%%%%%%%%%%%%%%%%%%%%%%%%%%%%%%%%%%%%%%%%%%%%%%%%%%%%%%%
\begin{theorem}
\label{thmB}
Let $s\in(\frac{1}{2},2]$. Then there exist some numbers $\rho >0$,  $\lambda >0$ and $C>0$ such that for any $u_0\in H^s_0(\T )$ with
$\|u_0\|_s <\rho$, there is a (unique) solution $u:\R ^+ \to H^s_0(\T )$ of \eqref{BOs} with $u\in Z_{s,T}$  for all $T>0$
and such that
\begin{equation}
\|u(t)\|_s \le C e^{-\lambda t} \|u_0\|_s \qquad \forall t\ge 0.
\end{equation}
\end{theorem}
%%%%%%%%%%%%%%%%%%%%%%%%%%%%%%%%%%%%%%%%%%%%%%%%%%%%%%%%%%%%%%%%%%
The proofs of Theorem \ref{thmA} and Theorem \ref{thmB} are given in the next sections.
\subsubsection{Linear Theory}
In this section, we focus on the well-posedness and the smoothing property of the linearized
BO equation with localized damping:
\begin{equation}
\label{R5}
u_t + {\mathcal H} u_{xx} + GDGu =0, \qquad u(0)=u_0.
\end{equation}
Let $s\in \R$ and let $Au=-({\mathcal H}u_{xx} +GDGu)$ with domain ${\mathcal D}(A)=H^{s+2}_0(\T )\subset H^s_0(\T )$. Our first result is the
%%%%%%%%%%%%%%%%%%%%%%%%%%%%%%%%%%%%%%%%%%%%%%%%%%%%%%%%%%%%%%%%%%
%%%%%%%%%%%%%%%%%%%%%%%%%%%%%%%%%%%%%%%%%%%%%%%%%%%%%%%%%%%%%%%%%%
\begin{lemma}
\label{lem100}
$A$ generates a continuous semigroup in $H^s_0(\T )$, denoted by $(S(t))_{t\ge 0}$.
\end{lemma}
\noindent
{\em Proof.} Let $C=C(s)$ be the constant in Claim 1. Clearly, $A-C$ is a densely defined closed operator in
$H^s_0(\T )$. Furthermore, by Claim 1,
\[
(Au -Cu,u)_s \le - \| D^{\frac{1}{2}}  (Gu)\| _s^2 \quad \forall u\in H^{s+2}_0(\T ),
\]
which shows that $A-C$ is dissipative. It is easily verified that $D(A^*)=D(A)=H^{s+2}_0(\T )$. Thus
\[
(A^*u-Cu,u)_s=(u,Au-Cu)_s\le 0\quad \forall u\in H^{s+2}_0(\T ),
\]
so that $A^*-C$ is dissipative too. Thus, $A-C$ generates a semigroup of contractions in $H^s_0(\T )$ by \cite[Cor. 4.4, p. 15]{pazy}.\qed

%We denote by $S(t)$ the semigroup generated by $A$.
%Using interpolation, we obtain that
%$S(t)H^s_0(\T )\subset H^s_0(\T )$ for all $t\ge 0$ and all $s\in [0,2]$ with
%\[
%\|S(t)u_0\|_{C([0,T],H^s_0(\T ))} \le C(T) \|u_0\|_s.
%\]
Now we turn our attention to the smoothing effect.
%%%%%%%%%%%%%%%%%%%%%%%%%%%%%%%%%%%%%%%%%%%%%%%%%%%%%%%
%%%%%%%%%%%%%%%%%%%%%%%%%%%%%%%%%%%%%%%%%%%%%%%%%%%%%%%
\begin{proposition}
\label{prop10}
Let $s\ge 0$, $v_0\in H^s_0(\T )$ and $g\in L^2(0,T,H^{s-\frac{1}{2}}_0(\T ))$. Then the solution $v$ of
\begin{equation}
\label{R10}
v_t +{\mathcal H} v_{xx} +GDG v=g, \qquad v(0)=v_0
\end{equation}
satisfies $v\in Z_{s,T}$  with
\begin{equation} \|v\|_{Z_{s,T}}
%=  ||v||_{L^\infty (0,T,H^s_0(\T ))} +||v||_{L^2(0,T,H^{s+\frac{1}{2}} _0(\T ) )}
\le C(s,T) \left(   ||v_0||_s + ||g||_{L^2(0,T,H^{s-\frac{1}{2}}(\T ))} \right) ,
\label{R12}
\end{equation}
$C(s,T)$ being nondecreasing in $T$.
\end{proposition}
%%%%%%%%%%%%%%%%%%%%%%%%%%%%%%%%%%%%%%%%%%%%%%%%%%%%%%%%
\noindent
{\em Proof.} Let us assume first that $s=0$. To have enough regularity in the computations, we assume that
$v_0\in H^2_0(\T )$  and that $g\in C([0,T],H^2_0(\T ))$, so that the solution $v$ of \eqref{R10} satisfies
$v\in C([0,T],H^2_0(\T))\cap C^1([0,T],H^0_0(\T ))$. We now proceed as in the proof of Proposition \ref{smoothing}.
We set $Lv=v_t + {\mathcal H} v_{xx}$, $f=-GDGv$, so that $L v= f+g$. Pick any $\varphi \in C^\infty (\T )$, and let
$Av=\varphi (x) v$. Then
\[
\vert \int_0^T ([L,A]v,v)\, dt \vert \le 2  \vert \int_0^T (f+g,\varphi v) dt \vert +\|\varphi \|_{L^\infty}
\big( \|v_0\|^2 + \|v(T)\|^2 \big) \cdot
\]
Scaling in \eqref{R10} by $v$ yields
\begin{eqnarray*}
\frac{1}{2}\|v(t) \|^2 + \int_0^t \|D^{\frac{1}{2}} Gv\|^2 d\tau
&=  & \frac{1}{2} \|v_0\|^2 +\int_0^t (g,v) d\tau \\
&\le& \frac{1}{2} \|v_0\|^2 +\int_0^T \|g\|_{-\frac{1}{2}} \|v\|_{\frac{1}{2}} dt.
\end{eqnarray*}
This yields
\begin{equation}
\label{R15}
\|v\|^2_{L^\infty (0,T,H^0(\T ))} +\int_0^T \|D^{\frac{1}{2}} (Gv)\|^2 d\tau
\le \frac{3}{2} \|v_0\|^2
+3\int_0^T \|g\|_{-\frac{1}{2}}\|v\|_{\frac{1}{2}} dt.
\end{equation}
Computations similar to those in Proposition \ref{smoothing}  give that
\begin{eqnarray*}
\vert \int _0^T (f+g,\varphi (x) v)\, d\tau \vert
&\le &  C\|\varphi \|_1 \int_0^T ( \|D^{\frac{1}{2}}(Gv)\|^2 +
\|v\|^2  + \|g\|_{-\frac{1}{2}} \|v\|_{\frac{1}{2}} ) dt \\
&\le& C(T, \|\varphi \|_1)
\left(
\|v_0\|^2 + \int_0^T \|g\|_{-\frac{1}{2}} \|v\|_{\frac{1}{2}} dt
\right) ,
\end{eqnarray*}
hence
\[
\vert \int_0^T ([L,A]v,v)dt \vert
\le C(T,\|\varphi \|_1)
\left( \|v_0\|^2 + \int_0^T \|g\|_{-\frac{1}{2}}\|v\|_{\frac{1}{2}}dt \right).
\]
Combined with \eqref{K1}-\eqref{K2}, the last inequality gives
\begin{equation}
\vert \int_0^T (\partial_x \varphi Dv,v)\,  dt \vert
\le C(T,\|\varphi \|_1)
\left( \|v_0\|^2 + \int_0^T \|g\|_{-\frac{1}{2}}\|v\|_{\frac{1}{2}}dt \right).
\label{K3}
\end{equation}
We pick again $b\in C^\infty _0(\omega )$, where $\omega =\{ x\in \T;\ a(x)>0\}$ and $x_0\in \T$.
Writing again $b^2(x)-b^2(x-x_0)=\partial _x \varphi$, we obtain successively, with \eqref{Z1} and  \eqref{R15}, that
\begin{eqnarray*}
\int_0^T \|D^{\frac{1}{2}} (bv)\|^2 dt &\le&
C (T) \left( \|v_0\|^2 + \int_0^T \|g\|_{-\frac{1}{2}}\|v\|_{\frac{1}{2}}dt \right),\\
\vert \int_0^T (b^2Dv,v) dt \vert &\le&
C\int_0^T \big( \|v\|^2 + \|D^{\frac{1}{2}} (bv)\|^2\big) dt \\
&\le& C(T)  \left( \|v_0\|^2 + \int_0^T \|g\|_{-\frac{1}{2}}\|v\|_{\frac{1}{2}}dt \right)
\end{eqnarray*}
and therefore, with \eqref{K3},
\begin{equation*}
\vert \int_0^T (b^2(x-x_0)Dv,v)dt \vert  \le
C(T)  \left( \|v_0\|^2 + \int_0^T \|g\|_{-\frac{1}{2}}\|v\|_{\frac{1}{2}}dt \right) .
\end{equation*}
Using a partition of unity, this yields
\[
\int_0^T \| v\|_{\frac{1}{2}} ^2 dt \le C(T)
\left(
\|v_0\|^2 + \int_0^T \|g\|^2 _{-\frac{1}{2}}dt
\right)
+\frac{1}{2}\int_0^T \|v\|^2_{\frac{1}{2}}dt.
\]
Combined with \eqref{R15}, this gives \eqref{R12} for $s=0$ when $v_0\in H^2_0(\T )$ and
$g\in C([0,T],H^2_0(\T ))$. This is also true for $v_0\in H^0_0(\T )$ and
$g\in L^1(0,T,H^{-\frac{1}{2}}_0(\T ))$ by density.

Let us now assume that $s\in (0,2]$. Pick  again any $v_0\in H^2_0(\T )$, $g\in C([0,T],H^2_0(\T ))$, and let
$v\in C([0,T],H^2_0(\T ))\cap C^1([0,T],H^0_0(\T ))$ denote the solution of \eqref{R10}. Set $w=D^s v$ and $h=D^s g$.
Note that
\[
D^s(GDGv) = GDG w + Ew
\]
with $E=[D^s,G]DGD^{-s} +GD [D^s,G]D^{-s}$. Note that $\|Ew\|\le C\|w\|$ and that $w$ solves
\[
w_t + {\mathcal H} w_{xx} +GDG w + Ew=h, \quad w(0)=w_0:=D^s v_0.
\]
Since
\begin{eqnarray*}
\vert \int_0^T (\varphi w,Ew)\, dt \vert
&\le&  C\|\varphi \|_1  \|w\|^2_{L^2 (0,T,H^0(\T ) ) } \\
&\le&  C(T , \|\varphi \|_1) \left(  \|w_0\|^2 + \int_0^T \| h \|_{-\frac{1}{2} }  \| w \| _{\frac{1}{2}}dt \right),
\end{eqnarray*}
we obtain in a similar fashion as above that
\[
\|w\|^2_{L^\infty (0,T,H^0(\T ))}  + \int_0^T \|  w\|_{\frac{1}{2}}^2 dt \le C(T)
\left( \|w_0\|^2 + \int_0^T \|h\|^2_{-\frac{1}{2}} dt \right),
\]
i.e.
\begin{equation}
\|v\|^2_{L^\infty (0,T,H^s(\T ))}  + \| v \|^2_{L^2(0,T,H^{s+\frac{1}{2}}(\T ))}
\le C(T)
\left( \|v_0\|_s^2 + \| g \|^2_{L^2(0,T,H^{s-\frac{1}{2}}(\T ))}  \right).
\label{R30}
\end{equation}
Inequality \eqref{R30} and the fact that $v\in C([0,T],H^s_0(\T ))$ are also true for
$v_0\in H^s_0(\T )$ and $g\in L^2(0,T, H^{s-\frac{1}{2}}_0(\T ))$ by density.\qed
\begin{corollary}
\label{cor10}
Let $s\ge 0$ and $B\in {\mathcal L} (H^s_0(\T ))$. Then for any  $v_0\in H^s_0(\T )$, the solution $v$ of
\begin{equation}
\label{G1}
v_t +{\mathcal H} v_{xx} +GDG v= Bv, \qquad v(0)=v_0
\end{equation}
fulfills $v\in Z_{s,T}$  with
\begin{equation}
%  ||v||_{L^\infty (0,T,H^s_0(\T ))}
%+||v||_{L^2(0,T,H^{s+\frac{1}{2}} _0(\T ))}
||u||_{Z_{s,T}} \le C(s,T) ||v_0||_s.
\label{G2}
\end{equation}
\end{corollary}
\begin{proof}
Since $A$ is the generator of a continuous semigroup on $H^s_0(\T )$ and $B$ is a bounded operator on
$H^s_0 (\T )$, $A+B$ is the generator of a continuous semigroup on $H^s_0(\T )$ (see e.g. \cite[Thm 1.1 p. 76]{pazy}).
Pick any $v_0\in H^s_0(\T )$, and let $v$ denote the solution of \eqref{G1} given by the semigroup generated by
$A+B$. Noticing that $g:=Bv\in C([0,T]; H^s_0(\T ))$, we infer from Proposition \ref{prop10}  that $v\in Z_{s,T}$ with
\[
  \|v\|_{L^\infty (0,T,H^s(\T ))}
+\|v\|_{L^2(0,T,H^{s+\frac{1}{2}} (\T ))}  \le
C(s,T) \left( \|v_0\|_s +\sqrt{T} \,\|B\|_{ {\mathcal L} (H^s_0(\T ))} \|v\|_{L^\infty (0,T;H^s(\T ))}\right) .
\]
Selecting $T_0>0$ such that $c(s,T_0) \sqrt{T_0} \|B\|_{ {\mathcal L} (H^s_0(\T ))} <1/2$ yields
\begin{equation}
  \|v\|_{L^\infty (0,T_0,H^s(\T ))}
+\|v\|_{L^2(0,T_0,H^{s+\frac{1}{2}} (\T ))}  \le
2 C(s,T_0)  \|v_0\|_s.
\label{G3}
\end{equation}
Successive applications of \eqref{G3} on the intervals $[0,T_0]$, $[T_0,2T_0]$,... give \eqref{G2} for any $T>0$.
\end{proof}

\subsubsection{Proof of Theorem \ref{thmA}}
Pick any $s\in (\frac{1}{2},2]$ and any $T>0$. Let $u_0\in H^s_0(\T )$. We write \eqref{BOs} in its integral form
\begin{equation}
\label{B32}
u(t)=S(t)u_0 -\int_0^t S(t-\tau )(uu_x)(\tau ) \, d\tau.
\end{equation}
Let $\Gamma (v)(t)=S(t)u_0 -\int_0^t S(t-\tau )(vv_x)(\tau )d\tau$. We have, by Proposition \ref{prop10}, that
\[
\| \Gamma (v) \|_{Z_{s,T}}  \le C \left(  \|u_0\|_s + \|  ( \frac{v^2}{2} )_x \|_{L^2(0,T,H^{s-\frac{1}{2}} (\T ))} \right)\cdot
\]
Clearly, for $u,v\in Z_{s,T}$,
\begin{eqnarray*}
\int_0^T \|(uv)_x\|^2_{s-\frac{1}{2}} dt
&\le& C\int_0^T \|uv\|^2_{s+\frac{1}{2}} dt \\
&\le& C\int_0^T \big(
   \|u\|^2_{L^\infty( \T )} \|v\|^2_{s+\frac{1}{2}}
+ \|u\|^2_{s+\frac{1}{2}} \|v\|^2_{L^\infty (\T )}
\big) dt \\
&\le& C\int_0^T \big( \|u\|^2_s\|v\|^2_{s+\frac{1}{2}} + \|u\|^2_{s+\frac{1}{2}} \|v\|^2_s \big) dt\\
&\le& C\left(
\|u\|^2_{ L^\infty (0,T,H^s(\T )) } \|v\|^2_{L^2(0,T,H^{s+\frac{1}{2}} (\T ))}  \right.  \\
&&\qquad  \left. +
\|v\|^2_{ L^\infty (0,T,H^s(\T )) } \|u\|^2_{L^2(0,T,H^{s+\frac{1}{2}} (\T ))}
\right) \\
&\le& C\|u\|_{Z_{s,T}}^2 \|v\|_{Z_{s,T}}^2,
\end{eqnarray*}
where we used the Sobolev embedding $H^s_0(\T )\subset L^\infty (\T )$ for $s>1/2$. Thus, there are some constants $C_0>0$ and $C_1>0$ such that
\begin{eqnarray*}
\| \Gamma (v) \|_{Z_{s,T}}                                     &\le& C_0 \|u_0\|_s + C_1 \| v \|^2_{Z_{s,T}} \qquad \forall v\in Z_{s,T} ,\\
\| \Gamma (v^1) -\Gamma (v^2) \|_{Z_{s,T}} &\le& C_1\big( \| v^1 \|_{Z_{s,T}} + \| v^2 \|_{Z_{s,T}}
\big) \| v^1 -v^2 \|_{Z_{s,T} } \qquad \forall v^1,v^2\in Z_{s,T}.
\end{eqnarray*}
Let $B=\{ v\in Z_{s,T}; \ \|v\|_{Z_{s,T}}  \le R \}$. We choose $R$ in such a way that $B$ is left invariant by $\Gamma$ and $\Gamma $
contracts in $B$, i.e.
\[
C_0\|u_0\|_s + C_1 R^2 \le R, \ \text{ and } 2C_1R<1.
\]
It is sufficient to take $R=(4C_1)^{-1}$ and $u_0\in H^s_0(\T )$ with
$\|u_0\|_s\le \rho := R/(2C_0)$. \qed

\subsubsection{Proof of Theorem \ref{thmB}}
We proceed as in \cite{PR}.
It has been proved that \eqref{BOs} is semi-globally exponentially stable in $H^0_0(\T )$. Obviously, the same analysis shows that
the {\em linearized} BO equation with localized damping is also exponentially stable in $H^0_0(\T )$, i.e.
\begin{equation}
\|S(t)u_0\| \le Ce^{-\lambda t} \|u_0\|
\label{Z100}
\end{equation}
for all $u_0\in H^0_0(\T )$ and some constants $C,\lambda >0$. If $u_0\in H^2_0(\T )$, then  $u(t)=S(t)u_0$ solves
\begin{equation}
\label{Z2}
u_t+{\mathcal H} u_{xx} +GDGu =0, \quad u(0)=u_0.
\end{equation}
Letting $v=u_t$, $v$ solves also
\begin{equation}
\label{Z3}
v_t + {\mathcal H} v_{xx} +GDG v = 0,\quad v(0)=v_0:= -({\mathcal H } u_{0,xx} +GDG u_0).
\end{equation}
\eqref{Z100} yields
\[
\|v(t)\| =\|S(t)v_0\| \le C e^{-\lambda t} \|v_0\|,
\]
and thus
\[
\|S(t) u_0\|_2 \le C' e^{-\lambda t} \|u_0\|_2.
\]
By interpolation, this shows that for any $s\in [0,2]$, for any $u_0\in H^s_0(\T )$  and for some constant $C>0$
(independent of $s$, $u_0$, and $t$), it holds
\begin{equation}
\label{Z5}
\|S(t) u_0\|_s \le C e^{-\lambda t} \|u_0\|_s.
\end{equation}
Let $s>1/2$ and  $u_0\in H^s_0(\T )$. For
\[u\in Z_{s,T}([n,n+1]):=C([n,n+1],H^s_0(\T )) \cap
L^2 (n,n+1,H_0^{s+\frac{1}{2}}(\T )), \]
let
\[
\tres u \tres_n =\|u\|_{L^\infty (n,n+1,H^s(\T )) } + \|u\|_{L^2(n,n+1,H^{s+\frac{1}{2}}(\T )) }\cdot
\]
Finally, let
\[
\|u\|_E =\sup_{n\ge 0} \big( e^{n\lambda} \tres u\tres_n \big) \le +\infty .
\]
Introduce the space
\[
E=\{ u\in C(\R ^+ , H^s_0(\T ))\cap L^2_{loc} (\R ^+, H_0^{s+\frac{1}{2}}(\T )); \ \|u\|_E <\infty \}.
\]
Endowed with the norm $\|\cdot \|_E $, $E$ is a Banach space. We search for a solution of \eqref{BOs} in a closed ball
$B=\{u\in E;\ \|u\|_E \le R \}$ as a fixed point of the map $\Gamma (v)(t)=S(t)u_0 - \int_0^t S(t-\tau ) (vv_x)(\tau ) d\tau$.
By \eqref{Z5}, we have
\begin{equation}
\label{Z7}
\|S(n) u_0\|_s \le Ce^{-n\lambda } \|u_0\|_s \quad \forall n\ge 0.
\end{equation}
Combined  with Proposition \ref{prop10}, this gives for some constant $C_0>0$
\begin{equation}
\label{Z8}
\tres  S(t)u_0 \tres_n \le C_0 \, e^{-n\lambda} \|u_0\|_s,
\end{equation}
hence
\begin{equation}
\label{Z10}
\| S(t)u_0\|_E \le C_0 \|u_0\|_s.
\end{equation}
On the  other hand, for any $u,v\in E$,
\[
\tres \int_0^t S(t-\tau ) [(uv)_x(\tau)] d\tau \tres_n \le I_1+I_2
\]
with
\begin{eqnarray*}
I_1 &=& \tres S(t-n) \int_0^n S(n-\tau ) [(uv)_x (\tau )]d\tau \tres_n,\\
I_2 &=& \tres \int_n^t S(t-\tau ) [(uv)_x (\tau )]d\tau \tres_n
\end{eqnarray*}
By \eqref{R12} and \eqref{Z8},
\begin{eqnarray*}
I_1 &\le & C\| \int_0^n S(n- \tau ) [(uv)_x (\tau )] d\tau \|_s \\
&\le& C\sum_{k=1}^n \|S(n-k) \int_{k-1}^k S(k-\tau ) [(uv)_x(\tau )] \|_s \\
&\le& C\sum_{k=1}^n e^{-(n-k)\lambda} \| \int_{k-1}^k S(k-\tau )
[(uv)_x(\tau ) ]d\tau  \|_s \\
&\le& C \sum_{k=1}^n e^{-(n-k)\lambda } \|(uv)_x\|_{L^2(k-1,k,H^{s-\frac{1}{2}} (\T ))} \\
&\le& C\sum_{k=1}^n e^{-(n-k)\lambda } \tres u \tres_{k-1} \tres v \tres_{k-1}  \\
&\le& Ce^{-n\lambda } \| u\|_E \|v\|_E .
\end{eqnarray*}
On the other hand
\[
I_2\le C\|(uv)_x\|_{L^2(n,n+1,H^{s-\frac{1}{2}} (\T ))} \le C e^{-2n\lambda } \|u\|_E \|v\|_E.
\]
We have proved that for some constant $C_1>0$
\[
\tres \int_0^t S(t-\tau ) [(uv)_x(\tau )] d\tau \tres_n  \le 2 C_1 e^{-n\lambda } \|u\|_E \|v\|_E,
\]
hence
\[
\|\int_0^t S(t-\tau ) [(uv)_x (\tau )] d\tau \|_E \le 2 C_1 \|u\|_E \|v\|_E.
\]
Thus
\begin{eqnarray*}
\|\Gamma (v)\|_E &\le & C_0\|u_0\|_s + C_1 \|v\|_E^2, \\
\|\Gamma (v^1)-\Gamma (v^2)\|_E &\le& C_1( \|v^1\|_E + \|v^2\|_E) \|v^1-v^2\|_E.
\end{eqnarray*}
It follows that $\Gamma$ contracts in the  ball $B=\{ u\in E;\  \|u\|_E \le R \}$ if
\begin{equation}
\label{Z20}
2C_1R<1, \ \text{ and } \  C_0\|u_0\|_s + C_1 R^2 \le R.
\end{equation}
Let $R=\gamma\rho$ ($\gamma$ and $\rho$ being determined later), and assume that $\|u_0\|_s\le \rho$. The conditions
become
\begin{equation}
\label{ZZZ}
2C_1\gamma \rho < 1 ,\text{ and } C_0+C_1\gamma ^2 \rho \le \gamma.
\end{equation}
Pick $\gamma =2C_0$ and $\rho >0$ sufficiently small so that  \eqref{ZZZ} holds. Then $\Gamma$ contracts in $B$.
Replacing $\rho$ by $\|u_0\|_s$, we see that  the fixed point $u=\Gamma (u)$ satisfies
\[
\|u\|_{L^\infty (n, n+1,H^s(\T ))} \le e^{-n\lambda } \| u \|_E \le e^{-n\lambda} \gamma \|u_0\|_s.
\]
It follows that
\[ \|u(t)\|_s  \le Ce^{-\lambda t} \|u_0\|_s \quad \forall t\ge 0\]
for some constant $C>0$, provided that  $\|u_0\|_s\le \rho$.
\qed
\section{Control of the Benjamin-Ono equation}
\null ~\\
Let again $a$ and $G$ be as in \eqref{defa} and \eqref{defG}, respectively. We now focus on the control properties of the full BO equation.
More precisely, we aim to prove the exact controllability  of the system
\begin{equation}
u_t + {\mathcal H} u_{xx} + uu_x =Gh, \qquad u(0)=u_0, \label{Y1}
\end{equation}
where $h$ is the control input. If the exact controllability of the linearized system is well known (cf. Theorem A), the exact controllability of
\eqref{Y1} is challenging, as the contraction mapping theorem cannot be applied directly to BO. To overcome that difficulty, we
incorporate the feedback $f=-DGu$ into the control input $h$  to obtain a strong enough smoothing effect to
apply the contraction principle.   Setting
\begin{equation}
\label{Y3}
h(t) = -DG u(t) +D^{\frac{1}{2}} k(t),
\end{equation}
we are thus led to investigate the controllability of the system
\begin{equation}
u_t + {\mathcal H} u_{xx} + GDGu + uu_x =GD^{\frac{1}{2}} k, \qquad u(0)=u_0.
\label{Y4}
\end{equation}
We shall derive the following local exact controllability result.
\begin{theorem}
\label{controllability}
Let $s\in (\frac{1}{2}, 2]$ and $T>0$. Then there exists $\delta >0$ such that for any $u_0,u_1\in H^s_0(\T )$ with
\begin{equation}
\label{Y5}
\|u_0\|_s \le \delta , \qquad \|u_1\|_s \le \delta ,
\end{equation}
one may find a control $k\in L^2(0,T,H^s_0(\T ))$ such that the system \eqref{Y4} admits a (unique) solution $u$ in the class $Z_{s,T}$ for which
$u(T)=u_1$.
\end{theorem}
The proof of Theorem \ref{controllability} is done in three steps. In the first step, we prove the exact controllability of the linearized
system
\begin{equation}
u_t + {\mathcal H} u_{xx} + GDGu =GD^{\frac{1}{2}} k, \qquad u(0)=u_0,
\label{Y6}
\end{equation}
in $L^2_0(\T )$. In the second step, we prove the exact controllability  of \eqref{Y6} in $H^s_0(\T)$ for all $s>0$ by following the same
approach as in \cite{RZ2009}. Finally, in the third part we derive the exact controllability of the full BO equation by  using the contraction
mapping theorem as e.g. in \cite{rosier97,RZ2007b,RZ2009}. Note that Theorem \ref{main3} follows at once from Theorem \ref{controllability}
by letting
\[
h = -DG u + D^{\frac{1}{2}} k \in L^2(0,T,H_0^{s-\frac{1}{2}} (\T )).
\]
{\em Proof of Theorem \ref{controllability}.}\\

\noindent
{\sc Step1. Exact controllability of \eqref{Y6}  in $H^0_0(\T )$.} \\
First, the solution of \eqref{Y4} belongs to $Z_{s,T}$ for $u_0\in H^s_0(\T )$  and $k\in L^2(0,T,H^s_0(\T ))$, according to
Proposition \ref{prop10}. The adjoint system reads
\begin{equation}
-v_t-{\mathcal H} v_{xx} +GDG v = 0, \qquad v(T)=v_T. \label{Y9}
\end{equation}
Scaling in \eqref{Y6} by $v$ yields
\begin{equation}
\int_{\T} uvdx\big\vert_0^T = \int_0^T\!\!\!\int_{\T} k  D^{\frac{1}{2}} (Gv)dxdt.
\label{Y11}
\end{equation}
The computations are fully justified when $u_0,v_T\in H^2_0(\T )$ and $k\in L^2(0,T,H^{\frac{5}{2}}_0(\T ))$, and next extended
to the case when $u_0,v_T\in H^0_0(\T )$ and $k\in L^2(0,T,H^0_0(\T ))$ by density. Following the classical duality approach, we are led
to prove the following observability inequality
\begin{equation}
\|v_T\|^2 \le C \int_0^T\!\!\! \int_{\T} | D^{\frac{1}{2}} (Gv) |^2 dxdt. \label{Y12}
\end{equation}
Once \eqref{Y12} is proved, the exact controllability of \eqref{Y6} follows by noticing that the operator $\Gamma \in {\mathcal L} (H^0_0(\T ))$
defined by $\Gamma (v_T) = u(T)$, where $u$ denotes the solution of \eqref{Y6} associated with $u_0=0$ and
$k=D^{\frac{1}{2}} (Gv)$ and $v$ denotes the solution of \eqref{Y9}, is onto by \eqref{Y12} and Lax-Milgram theorem.

Let us prove \eqref{Y12} by contradiction. If \eqref{Y12} is not true, then one can pick a sequence $(v^n_T)$ in $H^0_0(\T )$ such that
\begin{equation}
\label{Y13}
1=\|v^n_T\| ^2  > n \int_0^T\!\!\!\int_{\T} |D^{\frac{1}{2}} (Gv^n)|^2 dxdt ,
\end{equation}
where $v^n$ denotes the solution of \eqref{Y9} issued from $v_T=v_T^n$.

Multiplying each term in \eqref{Y9} by $tv^n$ and integrating by parts results in
\begin{equation}
\label{Y14}
\frac{T}{2} \|v_T^n\|^2 = \frac{1}{2} \int_0^T\!\!\! \int_{\T} |v^n|^2 dxdt + \int_0^T \!\!\! \int_{\T } t \, |D^{ \frac{1}{2}} (Gv^n) |^2 dxdt.
\end{equation}

Computations similar to those in the proof of Proposition \ref{prop10} (changing $t$ into $\tau :=T-t$) give
\begin{equation}
\label{Y15}
\|v^n\|_{L^2(0,T,H^{\frac{1}{2}}(\T ))} \le C \|v_T^n\| \cdot
\end{equation}
Thus, by \eqref{Y9} and \eqref{Y15}, $(v^n)$ is bounded in $L^2(0,T,H^{\frac{1}{2}}_0(\T )) \cap H^1(0,T,H^{-\frac{3}{2}} (\T) ) $.
By Aubin-Lions' lemma, a subsequence of $(v^n)$, still denoted by $(v^n)$, has a strong limit (say $v$) in $L^2(0,T,H^0_0(\T ))$.
It follows from \eqref{Y13} and \eqref{Y14} that $(v_T^n)$ is a Cauchy sequence in $H^0_0(\T )$, hence it has a strong limit (say $v_T$)
in $H^0_0(\T )$, with $\|v_T\|=1$.   By standard semigroup theory, $v^n$ converges in $C([0,T],H^0_0(\T ))$ to the solution of
\eqref{Y9} associated with $v_T$, which therefore agrees with $v$. By \eqref{Y13}, $D^{\frac{1}{2}} (Gv) \equiv 0$, hence $Gv\equiv 0$. We
conclude that $v$ satisfies
\begin{eqnarray*}
v_t + {\mathcal H} v_{xx} &=& 0,\\
Gv &=& 0.
\end{eqnarray*}
It follows from Proposition \ref{unique_continuation} that $v\equiv 0$. In particular $v_T=v(T)=0$, a property which contradicts the fact that
$\|v_T\|=1$. The proof of \eqref{Y12} is achieved. \\

\noindent
{\sc Step 2. Exact controllability of \eqref{Y6} in $H^s_0(\T )$.}\\
Picking any number $s>0$, we aim to prove the exact controllability of \eqref{Y6} in $H^s_0(\T )$.
Notice first that the system \eqref{Y9} is (backward) well-posed in $H^{-s}_0(\T)$, since the conclusion of
Lemma \ref{lem100} is still valid when ${\mathcal H}u_{xx}$  is replaced by  $-{\mathcal H}u_{xx}$ in \eqref{R5}.
Thus, the following estimate holds
\[
\| v\| _{L^\infty (0,T,H^{-s}(\T ))} \le C \| v_T\|_{-s}.
\]
On the other hand, setting $w=(1-\partial _x^2)^{-\frac{s}{2}} v$, we see that $w$ solves
\begin{eqnarray*}
&&-w_t-{\mathcal H} w_{xx} + GDG w = (1-\partial _x ^2 )^{-\frac{s}{2}} [(1-\partial _x ^2) ^{\frac{s}{2}} , GDG ]w =: B w\\
&& w(T)= (1-\partial _x ^2) ^{-\frac{s}{2}} v_T =: w_T.
\end{eqnarray*}
Note that $B\in {\mathcal L} (H^\sigma_0(\T ))$ for all $\sigma\in\R$ (see e.g. \cite{laurent}).
Using computations similar to those to prove Corollary \ref{cor10}, we see that
\begin{equation}
\label{AAA}
\|w \|_{L^2(0,T,H^{\frac{1}{2}}(\T ))} \le C \|w_T\|,
\end{equation}
and hence
\begin{equation}
\label{BBB}
\|v \|_{L^2(0,T,H^{-s+\frac{1}{2}}(\T ))} \le C \|v_T\|_{-s}.
\end{equation}

Assuming again that
$u_0=0$, we first note that \eqref{Y11} may be written
\[
\langle  v_T,u(T) \rangle _{-s,s} = \int_0^T \langle D^{\frac{1}{2}}  (Gv) ,  k  \rangle _{-s,s} dt,
\]
where $\langle \cdot ,\cdot  \rangle _{-s,s} $ denotes the duality pairing
$\langle \cdot ,\cdot  \rangle _{H^{-s}_0 (\T ), H^s_0(\T ) }$. We aim to prove the observability inequality
\begin{equation}
\label{Y20}
\|v_T\|^2_{-s} \le C \int_0^T \|Gv \|^2 _{-s+ \frac{1}{2}} dt.
\end{equation}
Once \eqref{Y20} is proved, the exact controllability of \eqref{Y6} in $H^s_0(\T )$ follows easily. Indeed, if
$\Gamma _{-s} \in {\mathcal L} (H^{-s}_0(\T ))$ is defined by $\Gamma _{-s}  (v_T) = (1-\partial _x^2 ) ^s u (T) $ where
$u$ solves \eqref{Y6} with $k=(1-\partial _x^2)^{-s} D^{\frac{1}{2}} (Gv)$ and $v$ still denotes the solution of \eqref{Y9}, then
\[
(v_T, \Gamma _{-s} (v_T) )_{-s} = \int_0^T \|D^{\frac{1}{2}} (Gv) \|^2 _{-s} dt \ge C\int_0^T  \|Gv\|^2_{-s+\frac{1}{2}} dt \ge C \|v_T\|^2 _{-s},
\]
so that $\Gamma _{-s}: H^{-s}_0(\T )\to H^{-s}_0(\T ) $ is onto. The same is true for the map $v_T\in H^{-s}_0(\T ) \to u(T) \in H^s_0(\T )$.
To prove \eqref{Y20}, we establish first a weaker estimate
\begin{equation}
\label{Y21}
\|v_T\|^2_{-s} \le C \left( \int_0^T \|Gv\|^2_{-s+\frac{1}{2}} dt + \|v_T\|^2_{-s-\frac{1}{2} } \right) .
\end{equation}
We argue again by contradiction. If \eqref{Y21} is false, then there is a sequence $(v_T^n)$ in $H^{-s}_0(\T )$ such that
\begin{equation}
\label{Y22}
1=\|v_T^n\| ^2_{-s}  > n
\left(
\int_0^T \|Gv^n\|^2_{-s+\frac{1}{2}} dt + \|v_T^n\| ^2 _{-s -\frac{1}{2}}
\right) .
\end{equation}
It follows that
\begin{eqnarray}
v_T^n &&\to 0 \quad  \text{ in } H_0^{-s-\frac{1}{2}} (\T ), \label{Y23} \\
v^n &&\to 0 \quad \text{ in } C([0,T],H^{-s-\frac{1}{2}}_0(\T )). \label{Y24}
\end{eqnarray}
Let $w^n=(1-\partial _x^2)^{-\frac{s}{2}} v^n$. Then $w^n$ solves
\begin{eqnarray*}
&&-w_t^n-{\mathcal H} w^n_{xx} + GDG w^n = (1-\partial _x ^2 )^{-\frac{s}{2}} [(1-\partial _x ^2) ^{\frac{s}{2}} , GDG ]w^n = B w^n,\\
&& w^n(T)= (1-\partial _x ^2) ^{-\frac{s}{2}} v^n_T =: w_T^n.
\end{eqnarray*}
Then $\|w_T^n\|=1$, $w_T^n \to 0$
in $H_0^{-\frac{1}{2}} (\T )$, and
\begin{eqnarray}
&& w^n\to 0\quad  \text{ in } C([0,T],H_0^{-\frac{1}{2}} (\T )) \label{Y25a}\\
&& \int_0^T \|G w^n\|^2_{\frac{1}{2}} dt \to 0. \label{Y25b}
\end{eqnarray}
For \eqref{Y25b}, we notice that
\begin{eqnarray}
\int_0^T \|G w^n\|^2_{\frac{1}{2}} dt
&=& \int_0^T \|G (1-\partial _x ^2)^{-\frac{s}{2}} v^n \|^2_{\frac{1}{2}} dt \nonumber \\
&\le& \int_0^T \|G v^n\|^2_{-s+\frac{1}{2}} dt + \int_0^T \|[G,(1-\partial _x^2)^{-\frac{s}{2}} ] v^n \|^2_{\frac{1}{2}} dt . \label{Y25c}
\end{eqnarray}
The first term in the right hand side of \eqref{Y25c} tends to $0$ by \eqref{Y22}. For the second one, we have that
\[
\int_0^T \|[G, (1-\partial _x ^2 ) ^{-\frac{s}{2}} ] v^n \|^2 _{\frac{1}{2}} dt \le C \int_0^T \|v^n\|^2_{-s-\frac{1}{2}} dt \le
C\|v^n\|^2_{L^\infty (0,T,H^{-s-\frac{1}{2}} (\T )  ) } \to 0,
\]
by \eqref{Y24}.

From \eqref{AAA}, we infer that
\begin{equation}
\label{Y26}
\|w^n\|_{L^2(0,T,H^{\frac{1}{2}}(\T ))} \le C \|w^n_T\|.
\end{equation}

Arguing as in Step 1 and using \eqref{Y26}, we can derive the following observability inequality
\[
\|w_T^n\|^2 \le C\left(
\int_0^T\!\!\! \int_{\T} |D^{\frac{1}{2}}(Gw^n)|^2 dxdt + \|Bw^n\|^2_{L^2(0,T,H^{-\frac{1}{2}}(\T ))}
\right) \cdot
\]
Combined with \eqref{Y25a} and \eqref{Y25b}, this yields $w_T^n\to 0$ in $H^0_0(\T )$, contradicting the fact that
$\|w_T^n\|=1$ for all $n$. The proof of \eqref{Y21} is complete. Finally, we prove \eqref{Y20} by contradiction.
If \eqref{Y20} is false, then there is a sequence $(v_T^n)$ in $H^{-s}_0(\T )$ such that
\begin{equation}
1=\|v_T^n\|^2_{-s} >n \int_0^T \|Gv^n\|^2 _{-s+\frac{1}{2}} dt.
\label{Y27}
\end{equation}
 Extracting a subsequence still denoted by $(v_T^n)$, we can assume that $(v_T^n)$ is strongly convergent
 in $H_0^{-s-\frac{1}{2} } (\T )$ by compactness of the embedding $H_0^{-s} (\T ) \subset H_0^{-s-\frac{1}{2} } (\T )$.
 Using  \eqref{Y27}, we infer from \eqref{Y21} that $(v_T^n)$ is also strongly convergent
 in $H_0^{-s} (\T )$. Its limit $v_T$ satisfies $\|v_T\|_{-s} = 1 $, and the solution $v$ of \eqref{Y9} satisfies $Gv=0$ by \eqref{Y27}.
 Thus for a.e. $t\in (0,T)$
 \[
 v_{xxx}(\cdot ,t ) = {\mathcal H } v_{xxx} (\cdot , t ) = 0 \qquad \text{ on } \omega .
 \]
 We conclude with Lemma \ref{Fourier} that $v\equiv 0$, hence $v_T=0$, which contradicts $\|v_T\|_{-s}=1$.
 The proof of \eqref{Y20} is achieved. \\

 \noindent
 {\sc Step 3. Fixed-point argument in $H^s_0(\T )$.}\\
We proceed as in \cite{rosier97}. Pick any $s\in (\frac{1}{2}, 2]$ and any $T>0$.
We still denote by $(S(t))_{t\ge 0}$ the semigroup introduced in Lemma \ref{lem100}
and by $Z_{s,T}$  the space introduced in \eqref{R2}. For $v\in Z_{s,T}$, we set
\[
\omega (v) = \int_0^T S(T-t) (vv_x)(t)\, dt.
\]
From Step 2 we know that the linearized system, namely
\eqref{Y6}, with initial data $u_0\in H^s_0(\T ) $ and control function  $k\in L^2(0,T , H^s_0(\T ) )$ is well-posed and exactly controllable
in $H^s_0(\T ) $.  By a classical functional analysis argument (see e.g. \cite[Lemma 2.48 p. 58]{coron-book}), one can construct a continuous operator
$\Lambda : H^s_0(\T )  \to L^2(0,T, H^s_0(\T ))$ such that for any $u_1\in H^s_0(\T ) $ the solution $u$ of \eqref{Y6} associated with
$u_0=0$ and $k=\Lambda (u_1)$  satisfies $u(T) = u_1$. Let us denote by $u=W(k)$ the corresponding trajectory. We know from Proposition
\ref{prop10}   that $W$ is continuous from $L^2(0,T,H^s_0(\T ))$ into $Z_{s,T}$. Let $u_0,u_1\in H^s_0(\T )$ be given with
\[
\|u_0\|_{H^s_0(\T )  } < \delta, \qquad  \|u_1\|_{H^s_0(\T ) } <\delta ,
\]
where $\delta >0$ will be chosen later. Let $v\in Z_{s,T}$. If we choose  $k=\Lambda (u_1-S(T)u_0 +\omega  (v))$, then
\[
S(t)u_0 - \int_0^t S(t-\tau ) (vv_x)(\tau ) d\tau + W ( k ) (t)
=
\left\{
\begin{array}{ll}
u_0 \quad &\text{\rm if } t=0;\\
u_1 &\text{\rm if } t=T.
\end{array}
\right.
\]
It suggests to consider the nonlinear map $v\to \Gamma (v)$, where
\[
\Gamma (v)(t) = S(t)u_0 - \int_0^t S(t-\tau ) (vv_x)(\tau ) d\tau + W (\Lambda ( u_1-S(T)u_0 +\omega (v) ) ) (t).
\]
The proof will be complete if we can show that this map has a fixed point in the space $Z_{s,T}$.
Using the estimates in the proof of Theorem \ref{thmA}, we see that
\[
\|\omega (v)\|_s \le C \| \int_0^t S(t-\tau ) (vv_x)( \tau ) d\tau \|_{Z_{s,T}} \le C \|v\|^2_{Z_{s,T}}
\]
and that there are some constants $C_0>0$ and $C_1>0$ such that
\begin{eqnarray*}
\| \Gamma (v)\|_{Z_{s,T}}                                     &\le& C_0 (\|u_0\|_s + \|u_1\|_s) + C_1 \| v \|^2_{Z_{s,T}} \qquad \forall v\in Z_{s,T},\\
\| \Gamma (v^1) -\Gamma (v^2) \|_{Z_{s,T}} &\le& C_1\big( \| v^1 \|_{Z_{s,T}} + \| v^2 \|_{Z_{s,T}}
\big) \| v^1 -v^2\|_{Z_{s,T}}  \qquad \forall v^1,v^2\in Z_{s,T}.
\end{eqnarray*}
Let $B=\{ v\in Z_{s,T}; \ \|v\|_{Z_{s,T}}  \le R \}$. We choose the radius $R$ in such a way that the ball $B$ is left
invariant by $\Gamma$ and $\Gamma $ contracts in $B$, i.e.
\[
C_0 (\|u_0\|_s + \|u_1\|_s) + C_1 R^2 \le R, \ \text{ and } 2C_1R<1.
\]
It is sufficient to take $R=(4C_1)^{-1}$ and
$\delta := R/(4C_0)$.
The proof of Theorem \ref{controllability} is complete.\qed

\section*{Acknowledgements}
The authors wish to thank Institut Henri Poincar\'e (Paris, France) for providing a very
stimulating environment during the ``Control of Partial Differential Equations and Applications''
program in the Fall 2010. LR was  partially supported by the Agence Nationale de la Recherche, Project CISIFS,
grant ANR-09-BLAN-0213-02. FL was partially supported by CNPq and FAPERJ/Brazil.

\end{document}